\documentclass[preprint,12pt]{elsarticle}
\usepackage{amsthm,amsmath,amsfonts,amssymb}

\def\wt#1{\widetilde{#1}}

\theoremstyle{definition}
\newtheorem{thm}{Theorem}
\newtheorem{lem}{Lemma}
\newtheorem{cor}{Corollary}
\newtheorem{Ex}{Example}
\newtheorem{rmk}{Remark}

\usepackage[sf,bf,tiny]{titlesec}
\titlelabel{\thetitle. \enspace}

\journal{Journal of Approximation Theory}

\begin{document}

\begin{frontmatter}
\title{Reproducing kernel orthogonal polynomials on the multinomial distribution.
}
\author[label1]{Persi Diaconis}
\address[label1]{Department of Statistics, Sequoia Hall, 
390 Serra Mall, Stanford University, Stanford, California 94305-4065, USA.}
\ead{diaconis@stanford.edu}
\author[label2]{Robert C. Griffiths}
\address[label2]{Department of Statistics, University of Oxford, 24-29 St Giles', Oxford, OX1 3LB, UK, Corresponding Author}
\ead{griff@stats.ox.ac.uk}

\end{frontmatter}
\section*{Abstract}
{\small 
Diaconis and Griffiths (2014)
study the multivariate Krawtchouk polynomials orthogonal on the multinomial distribution. 
 In this paper we derive the reproducing kernel orthogonal polynomials $Q_n(\boldsymbol{x},\boldsymbol{y};N,\boldsymbol{p})$ on the multinomial distribution which are sums of products of orthonormal polynomials in $\boldsymbol{x}$ and $\boldsymbol{y}$ of fixed total degree $n=0,1,\ldots , N$. 
$
\sum_{n=0}^N\rho^nQ_n(\boldsymbol{x},\boldsymbol{y};N,\boldsymbol{p})
$
arises naturally from a probabilistic argument.
An application to a multinomial goodness of fit test is developed, where the chi-squared test statistic is decomposed into orthogonal components
which test the order of fit.
 A new duplication formula for the reproducing kernel polynomials in terms of the 1-dimensional Krawtchouk polynomials is derived.
  The duplication formula allows a Lancaster characterization of all reversible Markov chains with a multinomial stationary distribution whose eigenvectors are multivariate Krawtchouk polynomials and where eigenvalues are repeated within the same total degree.
The $\chi^2$ cutoff time, and total variation cutoff time is investigated in such chains. 

Emphasis throughout the paper is on a probabilistic understanding of the polynomials and their applications, particularly to Markov chains.
\medskip

{\small
\noindent
\emph{Keywords:}
bivariate multinomial distributions; cutoff time; duplication formula; Ehrenfest urns; hypergroup property; Lancaster distributions; multinomial goodness of fit; multivariate Krawtchouk polynomials;
 reproducing kernel polynomials. 
}

\section{Introduction and summary}\label{section:1}
This section gives background on univariate Krawtchouk polynomials and an overview of the main results of the paper.

\citet{DG2012} study the connection between generalized Ehrenfest urn models, bivariate binomial distributions of the Lancaster type with Krawt\-chouk polynomial eigenfunctions, and sums of correlated Bernoulli random variables.
\citet{G1971} and \citet{DG2014} construct multivariate Krawtchouk polynomials orthogonal on the multinomial distribution and study their hypergroup properties.
\citet{G2016a} extends the multivariate Krawtchouk polynomials to multivariate orthogonal polynomials on a larger class of distributions and considers Lancaster expansions of them. \citet{G2016b} studies the role of the multivariate Krawtchouk polynomials in spectral expansions of transition functions of composition birth and death processes. These are multivariate versions of the 1-dimensional expansions of \citet{KMG1957a,KMG1957b,KMG1958}.
  Recent representations and derivations of orthogonality of these polynomials are in \citet{F2016a,F2016b,GVZ2013,GR2011,I2012,M2011}.
 \citet{ZL2009} show that these polynomials are eigenfunctions in classes of reversible composition Markov chains which have multinomial stationary distributions and use them to get sharp rates of convergence to stationarity.

Let $\{Q_n(x;N,p)\}_{n=0}^N$ be the Krawtchouk polynomials, orthogonal on the binomial $(N,p)$ distribution, scaled so that $Q_n(0;N,p)=1$ and
\[
\mathbb{E}\Bigl [Q_n(X;N,p)Q_m(X;N,p)\Bigr ] = \delta_{nm}h_n(p)^{-1},
\]
where, with $q=1-p$,
\[
h_n(p)={N\choose n}(p/q)^n.
\]
A  generating function is
\begin{equation}
\sum_{n=0}^N{N\choose n}t^nQ_n(x;N,p) = (1-tq/p)^x(1+t)^{N-x}.
\label{KGF:00}
\end{equation}
An explicit formula for the polynomials is
\[
Q_n(x;N,p)=
\sum_{\nu=0}^N \big (-q/p\big )^\nu
\frac{
{x\choose \nu}{N-x\choose n - \nu}
}
{ {N\choose n}}.
\]
If $X$ is a binomial random variable, the transform, applied in Theorem \ref{Duplicationthm} below, is
\begin{equation}
\mathbb{E}\big [\psi^XQ_n(X;N,p)\big ] = \big (q(1-\psi)\big )^n\big (p\psi+q\big )^{N-n}.
\label{KTR:0}
\end{equation}
Details of the Krawtchouk polynomials can be found in \citet{I2005}.
In the following we sometimes suppress the parameters $N,p$ and use the notation $Q_n(x)\equiv Q_n(x;N,p)$. An important \emph{hypergroup property} or \emph{duplication formula} for the Krawtchouk polynomials, proved by \citet{E1969}, with an alternative proof in \citet{DG2012}, is that (without loss of generality when $p \geq 1/2$) there is a probability distribution $\varphi_{xy}(z)$, $z=0,1,\ldots, N$ with parameters $x,y =0,1,\ldots, N$ such that
\begin{equation}
Q_n(x)Q_n(y) = \mathbb{E}_{\varphi_{xy}}\Bigl [Q_n(Z)\Bigr ].
\label{duplicate:0}
\end{equation}
The hypergroup property (\ref{duplicate:0}) is equivalent to
\begin{equation}
K(x,y,z) = \sum_{n=0}^Nh_nQ_n(x)Q_n(y)Q_n(z) \geq 0
\label{tripple:0}
\end{equation}
for $x,y,z = 0,1,\ldots ,N$. Then
\[
\varphi_{xy}(z)={N\choose z}p^zq^{N-z}K(x,y,z).
\]
There is a general theory of orthogonal functions which have a hypergroup property, \citet{BH2008}.

A bivariate binomial random vector $(X,Y)$ has a Lancaster probability distribution if 
\begin{equation}
p(x;y) = b(x;N,p)b(y;N,p)
\Bigl \{ 1 + \sum^N_{n=1}\rho_n h_nQ_n(x)Q_n(y)\Bigr \},
\label{L:0}
\end{equation}
for $x,y=0,1,\ldots ,N$. (For a general introduction to Lancaster expansions such as (\ref{L:0}) see \citet{L1969,K1996}.) \citet{E1969} characterized the correlation sequences $\{\rho_n\}^N_{n=0}$ such that (\ref{L:0}) is non-negative and thus a proper distribution as having a representation
$\rho_n = \mathbb{E}\big [Q_n(Z)\big ]$, 
where $Z$ is a random variable on $\{0,1,\ldots,N\}$.
\citet{DG2012} study further characterizations of distributions with this correlation sequence.

In a general multivariate orthogonal polynomial system the reproducing kernel polynomials $Q_n(\boldsymbol{x},\boldsymbol{y})$ are the sum of products of two orthonormal polynomials in $\boldsymbol{x}$ and $\boldsymbol{y}$ of total degree $n$, for  $n=0,1,\ldots$, \citet{DX2015}. These are known for several classical orthogonal polynomials systems (without being exhaustive): multivariate Jacobi polynomials; multivariate Hahn polynomials and multivariate Krawtchouk polynomials \citep{G1979,G2006,GS2013,X2015}. \citet{KZ2009} use several systems of reproducing kernel polynomials, including those on the multinomial, in Markov chain rate of convergence problems.

In this paper we study the reproducing kernel orthogonal polynomials on the multinomial.
These appear in a circulated note \citet{G2006} and are derived independently in \citet{X2015} as a limit from Hahn polynomial reproducing kernel polynomials. Their construction in this paper is different from that in \citet{X2015} and similar to \citet{G2006}.
 A new duplication formula, or hypergroup property, is found in Section \ref{duplication:10} which has the form
\[
Q_n(\boldsymbol{x},\boldsymbol{y};N,\boldsymbol{p}) = h_n(N,p)\mathbb{E}_{\varphi_{\boldsymbol{x},\boldsymbol{y}}}\bigl [Q_n(Z;N,p)\bigr ],
\]
where $Q_n(Z;N,p)$ is a 1-dimensional Krawtchouk polynomial and $Z$ is a random variable which has a distribution 
$\varphi_{\boldsymbol{x},\boldsymbol{y}}(z)$ depending on $\boldsymbol{x},\boldsymbol{y}$. This formula reduces to (\ref{duplicate:0}) if $d=2$. The mixing measure $\varphi_{\boldsymbol{x},\boldsymbol{y}}$ has an interesting probabilistic interpretation in terms of matching probabilities in two sequences of multinomial trials. 
 A duplication formula for multi-dimensional Jacobi and Hahn reproducing kernel polynomials in terms of the 1-dimensional Jacobi polynomials is derived in \citet{GS2013}. Different duplication formulae are derived in \citet{KS1997} and \citet{X2015}. The duplication formula for the reproducing kernel polynomials on the multinomial is an analogue of the Jacobi polynomial duplication formulae, though the calculations are  different.

Reproducing kernel polynomials on the $d$-dimensional product Poisson distribution are obtained in Section \ref{s:Poisson} as a limit from the multinomial reproducing kernel polynomials.

Bivariate multinomial distributions which have Lancaster expansions in reproducing kernel polynomials are characterized in Section \ref{s:extreme}. The extreme points of such distributions are described by transition distributions in Ehrenfest urns with balls of $d$ colours.
The $\chi^2$ cutoff time, and total variation cutoff time is investigated in such chains in Section \ref{s:cutoff}. This is a good motivation for considering their eigenfunction structure. 

A new application to a multinomial goodness of fit test is developed in Section \ref{s:chi}, where the chi-squared test statistic is decomposed into orthogonal components testing the order of fit of sub-samples of $n=1,2,\ldots,N$ from a multinomial sample with $N$ observations.
\section{Orthogonal polynomials on the multinomial distribution}\label{OPsection}
This section gives an overview of multivariate Krawtchouk polynomials, details and examples can be found in \citet{DG2014}.

Multivariate orthogonal polynomials are not unique because there are different linear combinations of variables which can be used and different possible orderings of polynomials of the same total degree in a Gram-Schmidt construction.

We give a brief description of properties of orthogonal polynomials on the multinomial distribution
\[
m(\boldsymbol{x},\boldsymbol{p}) = {N\choose \boldsymbol{x}}\prod_{j=1}^dp_j^{x_j},\>x_j \geq 0, \>j=1,\ldots, d,\> |\boldsymbol{x}|=N.
\]
An interpretation of the multinomial is that in an infinite population of $d$ coloured balls of relative frequencies $\boldsymbol{p}$, $m(\boldsymbol{x},\boldsymbol{p})$ is the probability of obtaining a configuration $\boldsymbol{x}$ in a sample of $N$.
Let $\{u^{(l)}\}_{l=0}^{d-1}$ be a complete set of orthonormal
functions on a probability distribution $\{p_j\}_{j=1}^d$, with
$u^{(0)} \equiv 1$, such that for $k,l=0,1,\ldots, d-1$, 
$
\sum_{j=1}^du_j^{(k)}u_j^{(l)}p_j = \delta_{kl}.
$
Define a collection of orthogonal polynomials
$\big \{Q_{\boldsymbol{n}}(\boldsymbol{x};\boldsymbol{u})\big \}$ with $\boldsymbol{n} = (n_1,\ldots n_{d-1})$
and $|\boldsymbol{n}| \leq N$ on the multinomial distribution
as the coefficients of $w_1^{n_1}\cdots
w_{d-1}^{n_{d-1}}$ in the generating function
\begin{equation}
G(\boldsymbol{x},\boldsymbol{w}; \boldsymbol{u}) 
= \prod_{j=1}^d\Big (1 + \sum_{l=1}^{d-1}w_lu_j^{(l)}\Big )^{x_j}.
\label{main_gf}
\end{equation}
It is straightforward to show, by using the generating function, that
\begin{equation}
\mathbb{E}\Big [Q_{\boldsymbol{m}}(\boldsymbol{X}; \boldsymbol{u})Q_{\boldsymbol{n}}(\boldsymbol{X};\boldsymbol{u})\Big ] = \delta_{\boldsymbol{m}\boldsymbol{n}}{N\choose \boldsymbol{n},N-|\boldsymbol{n}|}.
\label{normalizing}
\end{equation}
The transform of $Q_{\boldsymbol{n}}(\boldsymbol{x};\boldsymbol{u})$ with respect to the multinomial distribution is defined as
\begin{eqnarray}
&&\mathbb{E}\Big [\prod_{j=1}^ds_j^{X_j}Q_{\boldsymbol{n}}(\boldsymbol{X};\boldsymbol{u})\Big ]\nonumber \\
&&~~~~ = 
{N\choose \boldsymbol{n},N-|\boldsymbol{n}|}T_0(\boldsymbol{s})^{N-|\boldsymbol{n}|}T_1(\boldsymbol{s})^{n_1}\cdots T_{d-1}(\boldsymbol{s})^{n_{d-1}},
\label{OPMtransform}
\end{eqnarray}
where 
\[
T_i(\boldsymbol{s}) = \sum_{j=1}^{d}p_js_ju_j^{(i)},\>i=0,\ldots ,d-1.
\]
Let $Z_1,\ldots ,Z_N$ be independent identically distributed random
variables such that
\[
P(Z=k) = p_k, \>k=1,\ldots ,d.
\]
Then with $G$ defined at (\ref{main_gf}) and with 
\[
X_i = |\{Z_k: Z_k=i, k=1,\ldots, N\}|,
\]
\begin{equation}
G(\boldsymbol{X},\boldsymbol{w}; \boldsymbol{u}) 
= \prod_{k=1}^N\Big (1 + \sum_{l=1}^{d-1}w_lu_{Z_k}^{(l)}\Big ).
\label{symfnrep}
\end{equation}
In (\ref{symfnrep}) both sides are random variables.
From (\ref{symfnrep})
\begin{equation}
Q_{\boldsymbol{n}}(\boldsymbol{X};\boldsymbol{u}) = \sum_{\{A_l\}}\prod_{k_1\in A_1}
u_{Z_{k_1}}^{(1)}
\cdots \prod_{k_{d-1}\in A_{d-1}}
u_{Z_{k_{d-1}}}^{(d-1)},
\label{partitionrep}
\end{equation}
where summation is over all partitions into subsets of $\{1,\ldots
,N\}$, $\{A_l\}$ such that $|A_l| = n_l$, $l = 1,\ldots ,d-1$. 
That is, the orthogonal polynomials are symmetrized orthogonal functions in the tensor product set
\[
\bigotimes_{k=1}^N\big \{1,u^{(i)}_{Z_k}\big \}_{i=1}^{d-1}.
\]
The orthogonal polynomials could equally well be defined by (\ref{symfnrep}) or (\ref{partitionrep}) and the generating function (\ref{main_gf}) deduced.
Let
\[
S_l(\boldsymbol{X}) = \sum_{k=1}^Nu_{Z_k}^{(l)} = \sum_{j=1}^du_j^{(l)}X_j
,\> l =1,\ldots ,d-1.
\]
$Q_{\boldsymbol{n}}(\boldsymbol{X};\boldsymbol{u})$ is a polynomial of degree $|\boldsymbol{n}|$ in $(S_1(\boldsymbol{X}),\ldots ,S_{d-1}(\boldsymbol{X}))$
whose only term of maximal degree $|\boldsymbol{n}|$ is $\prod_1^{d-1}S_k^{n_k}(\boldsymbol{X})$.
\citet{DG2014} show that sometimes there is a duplication formula for the multivariate Krawtchouk polynomials similar to (\ref{duplicate:0}) which is inherited from the elementary basis $\boldsymbol{u}$.
If $u_d^{(i)} \ne 0$ let $\wt{u}^{(i)}= u^{(i)}/u^{(i)}_d$ be an orthogonal basis on $\bf{p}$. ($d$ does not have a particular significance, it could be another index instead.) Scale
$\wt{Q}(\boldsymbol{x};\boldsymbol{u}) = Q(\boldsymbol{x};\boldsymbol{u})/Q(N\boldsymbol{e}_d;\boldsymbol{u})$, where 
$\boldsymbol{e}_d = (\delta_{id})$. Then there exists a random vector $\boldsymbol{Z}_{\boldsymbol{x}\boldsymbol{y}}$ whose distribution depends on $(\boldsymbol{x},\boldsymbol{y})$ such that
\[
\wt{Q}_{\boldsymbol{n}}(\boldsymbol{x};\boldsymbol{u})\wt{Q}_{\boldsymbol{n}}(\boldsymbol{y};\boldsymbol{u})=\mathbb{E}\big [\wt{Q}_{\boldsymbol{n}}(\boldsymbol{Z}_{\boldsymbol{x}\boldsymbol{y}};\boldsymbol{u})\big ]
\]
if and only if there exists a random variable $\zeta_{jk}$ whose distribution depends on $(j,k)$ such that
\[
\wt{u}^{(i)}_j\wt{u}^{(i)}_k=\mathbb{E}\big [\wt{u}^{(i)}_{\zeta_{jk}}\big ].
\]

The Krawtchouk polynomials diagonalize the joint distribution of marginal counts in a contingency table.
This will be used in the proof of Theorem \ref{Thm1} of Section \ref{RKP:1}.
 Suppose $N$ observations are placed independently into an $r\times c$ table ($r \leq c$) with the probability of an observation falling in cell $(i,j)$ being $p_{ij}$. Denote the marginal distributions as $p_i^r = \sum_{j=1}^cp_{ij}$ and $p_j^c = \sum_{i=1}^rp_{ij}$. Let $p_{ij}$ have a Lancaster expansion (which is always possible, even for non-exchangeable $p_{ij}$)
\begin{equation}
p_{ij} = p^r_ip^c_j\big \{1 + 
\sum_{k=1}^{r-1}\rho_k u^{(k)}_iv^{(k)}_j\big \},
\label{elLancaster:0}
\end{equation}
where $\boldsymbol{u}$ and $\boldsymbol{v}$ are orthonormal function sets on $\boldsymbol{p}^r$ and $\boldsymbol{p}^c$. $\boldsymbol{u}$ is an orthonormal basis for functions $\{f(i);i=1,\ldots,r\}$ which have a finite variance under $\{p_i^r\}_{i=1}^r$. If $c > r$ there is a set of $c-r$ orthonormal functions $\boldsymbol{v}^*$ such that 
$\boldsymbol{v}^\prime = \boldsymbol{v} \oplus \boldsymbol{v}^*$ is an orthonormal basis for functions $\{g(j),j=1,\ldots,c\}$ which have a finite variance under $\{p_j^c\}_{j=1}^c$.
The eigenvalues $\{\rho_k\}$ may be complex, and are bounded in modulus by 1.

 Let $N_{ij}$ be the number of observations falling into cell $(i,j)$ and 
$X_i = \sum_{j=1}^cN_{ij}$, $Y_j = \sum_{i=1}^rN_{ij}$ the marginal counts. Then
\begin{eqnarray}
&&P\big (\boldsymbol{X}=\boldsymbol{x},\boldsymbol{Y}=\boldsymbol{y}\big ) = 
m(\boldsymbol{x};N,\boldsymbol{p}^r)m(\boldsymbol{y};N,\boldsymbol{p}^c)\nonumber \\
&&~~\times
\Big \{1 + \sum_{\boldsymbol{n}}\rho_1^{n_1}\ldots \rho_{r-1}^{n_{r-1}}{N\choose \boldsymbol{n}}^{-1}Q_{\boldsymbol{n}}(\boldsymbol{x};N,\boldsymbol{p}^r,\boldsymbol{u})Q_{\boldsymbol{n}^\prime}(\boldsymbol{y};N,\boldsymbol{p}^c,\boldsymbol{v}^\prime)\Big \},
\nonumber \\
\label{contingency:0}
\end{eqnarray}
where $\boldsymbol{n}^\prime$ has the same first $r$ elements as $\boldsymbol{n}$ and the last $c-r$ elements zero.
\citet{AG1935} showed (\ref{contingency:0}) for a $2\times 2$ table with the usual 1-dimensional Krawtchouk polynomials and \citet{G1971} for $r\times c$ tables.
\section{Reproducing kernel polynomials.} \label{RKP:1}
This section defines reproducing kernel polynomials and computes their Poisson kernel.  Results are extended via a limit to reproducing kernel polynomials on the  product Poisson distribution. 
A new duplication formula for the reproducing kernel polynomials in terms of the 1-dimensional Krawtchouk polynomials is derived in Section \ref{duplication:10}.
A statistical application for testing goodness of fit to a multinomial distribution is given.

Let $\big \{Q^\circ_{\boldsymbol{n}}(\boldsymbol{x})\big \}$ be a multivariate orthonormal polynomial set on a discrete probability distribution.
The reproducing kernel polynomials are defined as the sum of products of polynomials of the same total degree
\begin{equation}
Q_{n}(\boldsymbol{x},\boldsymbol{y}) = \sum_{|\boldsymbol{n}|=n}Q_{\boldsymbol{n}}^\circ(\boldsymbol{x})Q^\circ_{\boldsymbol{n}}(\boldsymbol{y}).
\end{equation}
It is clear that
\begin{equation}
Q_{n}(\boldsymbol{x},\boldsymbol{y}) = \sum_{|\boldsymbol{n}|=n}R_{\boldsymbol{n}}(\boldsymbol{x})R_{\boldsymbol{n}}(\boldsymbol{y})
\end{equation}
for \emph{any} orthonormal polynomial set $\{R_{\boldsymbol{n}}(\boldsymbol{x})\}$ on the same distribution, because it is possible to make an orthogonal transformation within polynomials of the same total degree that leaves 
$Q_{n}(\boldsymbol{x},\boldsymbol{y})$ invariant. Let 
\[
\xi_{\boldsymbol{n}}(\boldsymbol{x}) = \sum_{\boldsymbol{m}:|\boldsymbol{m}|=|\boldsymbol{n}|}a_{\boldsymbol{n}\boldsymbol{m}}\prod x_i^{m_i}
\]
be the leading terms of orthonormal polynomials $Q^\circ_{\boldsymbol{n}}(\boldsymbol{x})$.
Then if $|\boldsymbol{m}|=|\boldsymbol{n}|$
\[
\delta_{\boldsymbol{m}\boldsymbol{n}} = \mathbb{E}\big [Q^\circ_{\boldsymbol{m}}(\boldsymbol{X})Q^\circ_{\boldsymbol{n}}(\boldsymbol{X}) \big ]
= \mathbb{E}\big [Q^\circ_{\boldsymbol{m}}(\boldsymbol{X})\xi_{\boldsymbol{n}}(\boldsymbol{X})\big ].
\]
It follows that a set of multivariate orthonormal polynomials is always determined by their leading terms and reproducing kernel polynomials because
\[
Q^\circ_{\boldsymbol{n}}(\boldsymbol{x})
 = \mathbb{E}\big [Q^\circ_{\boldsymbol{n}}(\boldsymbol{Y})Q_{|\boldsymbol{n}|}(\boldsymbol{x},\boldsymbol{Y})\big ]
 = \mathbb{E}\big [\xi_{\boldsymbol{n}}(\boldsymbol{Y})Q_{|\boldsymbol{n}|}(\boldsymbol{x},\boldsymbol{Y})\big ].
 \]
 Another property of reproducing kernel polynomials is an expansion in mean square for a function $f(\boldsymbol{x})$ such that $\mathbb{E}\big [f(\boldsymbol{X})^2\big ] < \infty$ as
 \[
 f(\boldsymbol{x}) = \sum_n\mathbb{E}\big [f(\boldsymbol{Y})Q_n(\boldsymbol{x},\boldsymbol{Y})\big ].
 \]
Letting $K(\boldsymbol{x},\boldsymbol{y}) = \sum_{n=0}^NQ_n(\boldsymbol{x},\boldsymbol{y})$, $K$ is a reproducing kernel for the Hilbert space of square integrable functions. It has the property
\[
f(\boldsymbol{x}) = \mathbb{E}[f(\boldsymbol{Y})K(\boldsymbol{x},\boldsymbol{Y})],
\]
however $K(\boldsymbol{x},\boldsymbol{y})$ is concentrated at a single line where $\boldsymbol{x}=\boldsymbol{y}$ because
\begin{eqnarray*}
K(\boldsymbol{x},\boldsymbol{y}) &=& \sum_{n=0}^NQ_n(\boldsymbol{x},\boldsymbol{y})
\nonumber \\
&=& \sum_{\boldsymbol{n}}Q^\circ_{\boldsymbol{n}}(\boldsymbol{x})
Q^\circ_{\boldsymbol{n}}(\boldsymbol{y})
\nonumber \\
&=& \delta_{\boldsymbol{x},\boldsymbol{y}}P(\boldsymbol{Y}=\boldsymbol{y})^{-1}.
 \end{eqnarray*}
The last line follows because $\Big (\sqrt{P(\boldsymbol{Y}=\boldsymbol{y})}Q^\circ_{\boldsymbol{n}}(\boldsymbol{y})\Big )$ is an orthogonal matrix, indexed by $(N-|\boldsymbol{n}|,\boldsymbol{n}),\boldsymbol{y}$.
For general background on reproducing Kernel Hilbert spaces see \citet{BTA2004}.
 
In this paper our interest is in the reproducing kernel polynomial constructed from the orthonormal multivariate Krawtchouk polynomials
\[
Q^\circ_{\boldsymbol{n}}(\boldsymbol{x};\boldsymbol{u}) = \frac{Q_{\boldsymbol{n}}(\boldsymbol{x};\boldsymbol{u})}{\sqrt{\mathbb{E}\Big[Q_{\boldsymbol{n}}(\boldsymbol{X};\boldsymbol{u})^2\Big ]}}.
\]
Of course there are other such sets of orthogonal polynomials on the multinomial constructed in different ways from $\{Q_{\boldsymbol{n}}(\boldsymbol{x};\boldsymbol{u})\}$, however the reproducing kernel polynomials are invariant under which ever set is used.

The distribution of the marginal counts in a contingency table (\ref{contingency:0}) leads to a Poisson kernel for the reproducing kernel polynomials which is also a generating function for $\{Q_{n}(\boldsymbol{x},\boldsymbol{y};N,\boldsymbol{p})\}$. An explicit form for the polynomials is then obtained from this generating function. Classically, the Poisson kernel associated to an orthonormal family is used to represent a harmonic function in a domain in terms of its boundary values. Here, the analog is the generating function (\ref{RKPoisson}) defined below. In a general context a Poisson kernel being positive allows a construction of an interesting class of homogeneous reversible continuous time Markov processes $\{X(t)\}_{t \geq 0}$ with a stationary distribution $f$, based on transition functions of $X(t)$ given $X(0)=x$ of
\[
f(y;x,t) = f(y) \big \{ 1 + \sum_{n\geq 1}e^{-nt}\xi_n(x)\xi_n(y)\big \},
\]
where $\{\xi_n\}_{n\geq 0}$ is a complete set of orthonormal functions on $f$. The class of processes is constructed by subordination of $\{X(t)\}_{t\geq 0}$,
\[
{\cal C}_X = \Big \{\{\widetilde{X}(t)\}_{t\geq 0}:
\{\widetilde{X}(t)\}_{t\geq 0} = \{X(Z(t))\}_{t\geq 0}
~\{Z(t)\}_{t\geq 0}\text{~is~a~subordinator}\Big \}.
\]
${\cal C}_X$ is closed under subordination. That is, if $\{\widetilde{X}(t)\}_{t\geq 0} \in {\cal C}_X$ and $\{\widetilde{Z}(t)\}_{t\geq 0}$ is a subordinator, then 
 $\{\widetilde{X}(\widetilde{Z}(t))\}_{t\geq 0} \in {\cal C}_X$.

\medskip

\noindent
\begin{thm} \label{Thm1}

\noindent
\emph{(a)} The Poisson kernel
\begin{eqnarray}
&&1 + \sum_{n=1}^{N}\rho^nQ_n(\boldsymbol{x},\boldsymbol{y};N,\boldsymbol{p})\nonumber \\
&&=m(\boldsymbol{x};N,\boldsymbol{p})^{-1}m(\boldsymbol{y};N,\boldsymbol{p})^{-1}
\sum_{\boldsymbol{z}\leq \boldsymbol{x},\boldsymbol{y}}\rho^{|\boldsymbol{z}|}(1-\rho)^{N-|\boldsymbol{z}|}
{N\choose\boldsymbol{z},N-|\boldsymbol{z}|}
\nonumber \\
&&~~~~\times
\prod_{i=1}^dp_i^{z_i}
{N-|\boldsymbol{z}|\choose \boldsymbol{x}-\boldsymbol{z}}\prod_{i=1}^dp_i^{x_i-z_i}
{N-|\boldsymbol{z}|\choose \boldsymbol{y}-\boldsymbol{z}}\prod_{i=1}^dp_i^{y_i-z_i},
\label{RKPoisson}
\end{eqnarray}
which is non-negative if 
\begin{equation}
-\frac{1}{\frac{1}{\min_i p_i}-1}\leq \rho \leq 1.
\label{RKcond:0}
\end{equation}

\noindent
\emph{(b)} The reproducing kernel polynomials
\begin{eqnarray}
Q_{n}(\boldsymbol{x},\boldsymbol{y};N,\boldsymbol{p})
&=& \sum_{\boldsymbol{z}\leq \boldsymbol{x},\boldsymbol{y}: |\boldsymbol{z}| \leq n}
{N\choose |\boldsymbol{z}|}{N-|\boldsymbol{z}|\choose n - |\boldsymbol{z}|}(-1)^{n-|\boldsymbol{z}|}
\nonumber \\
&&~~~~\times
{|\boldsymbol{z}|\choose \boldsymbol{z}}
\frac{
\prod_{j=1}^d {x_j}_{[z_j]}{y_j}_{[z_j]}p_j^{-z_j}
}
{
N_{[|\boldsymbol{z}|]}^2
}
\label{Kernel_poly}
\end{eqnarray}
Notation $x_{[k]}=x(x-1)\cdots (x-k+1)$ and $x_{(k)} = x(x+1)\cdots (x+k-1)$ is used  in this paper. $x_{[0]} = x_{(0)} =1$, even if $x=0$.

\noindent\emph{(c)}
The transform of
$Q_n(\boldsymbol{x},\boldsymbol{y};N,\boldsymbol{p})$ for independent multinomial vectors $\boldsymbol{X},\boldsymbol{Y}$ is
\begin{eqnarray}
&&\mathbb{E}\Big [\prod_{i,j=1}^ds_i^{X_i}t_j^{Y_j}Q_{n}(\boldsymbol{X},\boldsymbol{Y};N,\boldsymbol{p})\Big ]
\nonumber\\
&&~~={N\choose n} \Big [T_0(\boldsymbol{s})T_0(\boldsymbol{t})\Big ]^{N-n}
\cdot 
\Big [\sum_{j=1}^dp_js_jt_j - T_0(\boldsymbol{s})T_0(\boldsymbol{t})
\Big ]^n,
\label{Q_transform}
\end{eqnarray}
where $T_0(\boldsymbol{s}) = \sum_{i=1}^ds_ip_i$ and similarly for $T_0(\boldsymbol{t})$.
\end{thm}

\begin{proof}

\noindent{(a)}
 Consider a $d\times d$ contingency table where the probability of an observation falling in cell $(i,j)$ is
\begin{eqnarray}
p_{ij} &=& p_ip_j\big \{1 -\rho + \delta_{ij}\rho p_i^{-1}\big \}\nonumber \\
&=& p_ip_j\big \{1 + \sum_{r=1}^{d-1}\rho u^{(r)}_iu^{(r)}_j\big \}
\label{pdiag:0}
\end{eqnarray}
for any orthonormal basis $\{u^{(r)}\}_{r=0}^{d-1}$. Notice that $p_{ij} \geq 0$ for all $i,j$ if and only if (\ref{RKcond:0}) holds.
Since 
$\rho_1=\rho_2=\ldots \rho_{d-1}=\rho$ in (\ref{contingency:0}) the joint distribution of the marginal counts $(\boldsymbol{X},\boldsymbol{Y})$  is 
\begin{eqnarray}
P\big (\boldsymbol{X} = \boldsymbol{x}, \boldsymbol{Y}=\boldsymbol{y}\big )&=&m(\boldsymbol{x};N,\boldsymbol{p})m(\boldsymbol{y};N,\boldsymbol{p})
\big \{1 + \sum_{k=1}^{N}\rho^kQ_k(\boldsymbol{x},\boldsymbol{y};N,\boldsymbol{p})\big \}.
\label{expansion:0}
\end{eqnarray}
Another expression is obtained from a direct probability calculation. The joint \emph{pgf} of $(\boldsymbol{X},\boldsymbol{Y})$ is
\begin{eqnarray}
\mathbb{E}\big [\prod_{i,j=1}^ds_i^{X_i}t_j^{Y_j}\big ]
&=& \Big (\sum_{i,j=1}^dp_{ij}s_it_j\Big )^N\nonumber \\
&=& \Big ((1-\rho)\big (\sum_{i=1}^dp_is_i\big )\big (\sum_{j=1}^dp_jt_j\big ) + \rho \sum_{i=1}^dp_is_it_i\Big )^N.
\label{pgf:0}
\end{eqnarray}
The coefficient of $\prod_{i,j=1}^ds_i^{X_i}t_j^{Y_j}$ in (\ref{pgf:0}) is
\begin{eqnarray}
P\big (\boldsymbol{X} = \boldsymbol{x}, \boldsymbol{Y}=\boldsymbol{y}\big ) &=&
\sum_{\boldsymbol{z}\leq \boldsymbol{x},\boldsymbol{y}}
{N\choose\boldsymbol{z},N-|\boldsymbol{z}|
}\rho^{|\boldsymbol{z}|}(1-\rho)^{N-|\boldsymbol{z}|}
\nonumber \\
&&~\times\prod_{i=1}^dp_i^{z_i}
{N-|\boldsymbol{z}|\choose \boldsymbol{x}-\boldsymbol{z}}\prod_{i=1}^dp_i^{x_i-z_i}
{N-|\boldsymbol{z}|\choose \boldsymbol{y}-\boldsymbol{z}}\prod_{i=1}^dp_i^{y_i-z_i},
\nonumber \\
\label{expansion:00}
\end{eqnarray}
where the diagonal counts are $\boldsymbol{z}$. Equating (\ref{expansion:0}) and (\ref{expansion:00}) gives (\ref{RKPoisson}).\\

\noindent{(b)}
The coefficient of
$\rho^n$ in (\ref{expansion:00}) evaluates to
\begin{equation}
\sum_{\boldsymbol{z}\leq \boldsymbol{x},\boldsymbol{y}:|\boldsymbol{z}|\leq n}
{N\choose |\boldsymbol{z}|}{N-|\boldsymbol{z}|\choose n - |\boldsymbol{z}|}(-1)^{n-|\boldsymbol{z}|}
{|\boldsymbol{z}|\choose \boldsymbol{z}}\prod_{j=1}^dp_j^{x_j+y_j-z_j}
{N - |\boldsymbol{z}|\choose \boldsymbol{x}-\boldsymbol{z}}
{N - |\boldsymbol{z}|\choose \boldsymbol{y}-\boldsymbol{z}}.
\label{coeftrho}
\end{equation}
Dividing (\ref{coeftrho}) by $m(\boldsymbol{x},\boldsymbol{p})m(\boldsymbol{y},\boldsymbol{p})$ and simplifying yields
(\ref{Kernel_poly}).

\noindent{(c)}
The transform of $Q_n(\boldsymbol{x},\boldsymbol{y};N,\boldsymbol{p})$ is the coefficient of $\rho^n$ in (\ref{pgf:0}) which is (\ref{Q_transform}).
\end{proof}
\begin{rmk}
A Markov chain can be constructed with transition functions 
\begin{eqnarray*}
P\big ( \boldsymbol{Y}=\boldsymbol{y}\mid \boldsymbol{X} = \boldsymbol{x}\big )&=&m(\boldsymbol{y};N,\boldsymbol{p})
\times\Big \{1 + \sum_{k=1}^{N}\rho^kQ_k(\boldsymbol{x},\boldsymbol{y};N,\boldsymbol{p})\Big \}
\end{eqnarray*}
connected to the Poisson kernel. The state space of the chain is the configuration of $N$ balls of $d$ colours $1,2\ldots,d$ in an urn. In a transition from a configuration $\boldsymbol{x}$, the $N$ balls are chosen without replacement from the urn to form $\boldsymbol{y}$. If a ball drawn is of type $i$ then it remains of type $i$ with probability $\rho$, or with probability $1-\rho$ its type is chosen to be $j$ with probability $p_j$, $j\in [d]$.
To see this consider the conditional \emph{pgf} of $\boldsymbol{Y}\mid \boldsymbol{x}$ from (\ref{pgf:0}). By inversion with respect to $\boldsymbol{s}$ the \emph{pgf} is equal to
\[
\prod_{i=1}^d\Big (\rho t_i+ (1-\rho)\sum_{j=1}^dp_jt_j\Big )^{x_i},
\]
giving the desired interpretation.
\end{rmk}
\begin{rmk}
The probability expression (\ref{expansion:0}) for the marginal distributions in a contingency table when $p_{ij}$ is given by (\ref{pdiag:0}) which is then used as a generating function for the reproducing kernel polynomials is a useful idea which is important. As a corollary a recursion is found using this representation for the generating function.
\end{rmk}
\begin{cor}
For $k\geq 1$ a recursive equation in $N$ for the reproducing kernel polynomials is
\begin{eqnarray}
&&Q_k(\boldsymbol{x},\boldsymbol{y};N,\boldsymbol{p})
= \sum_{i,j=1}^d\frac{x_i}{N}\cdot\frac{y_j}{N}\Bigg (
Q_k(\boldsymbol{x}-\boldsymbol{e}_i,\boldsymbol{y}-\boldsymbol{e}_j;N-1,\boldsymbol{p})
\nonumber \\
&&~~~~~~~~~~~~~~~~~~+\delta_{ij}\frac{p_i-p_ip_j}{p_ip_j}
Q_{k-1}(\boldsymbol{x}-\boldsymbol{e}_i,\boldsymbol{y}-\boldsymbol{e}_j;N-1,\boldsymbol{p})
\Bigg ).
\nonumber \\
\label{recurseN:0}
\end{eqnarray}
\end{cor}
\begin{proof}
Consider the marginal distributions of $\boldsymbol{X},\boldsymbol{Y}$ in $N$ multinomial trials with $\boldsymbol{p}$ given by (\ref{pdiag:0}).
Then partitioning the event that $\boldsymbol{X}=\boldsymbol{x},\boldsymbol{Y}=\boldsymbol{y}$ according to the classification of the last trial
\begin{eqnarray}
&&P(\boldsymbol{X}=\boldsymbol{x},\boldsymbol{Y}=\boldsymbol{y};N,\boldsymbol{p})
\nonumber \\
&&= \sum_{i,j=1}^d
p_ip_j\{1-\rho + \delta_{ij}\rho p_i^{-1}\}
P(\boldsymbol{X}=\boldsymbol{x}-\boldsymbol{e}_i,\boldsymbol{Y}=\boldsymbol{y}-\boldsymbol{e}_j;N-1,\boldsymbol{p}).
\nonumber \\
\label{recurseN:1}
\end{eqnarray}
The recursion (\ref{recurseN:0}) follows from equating coefficients of $\rho^k$ on both sides of (\ref{recurseN:1}) in view of (\ref{expansion:0}).
\end{proof}

\begin{rmk} \label{polys:300}
The first three reproducing kernel polynomials are:
\begin{eqnarray}
Q_0(\boldsymbol{x},\boldsymbol{y};N,\boldsymbol{p}) &=& 1,
\nonumber \\
Q_1(\boldsymbol{x},\boldsymbol{y};N,\boldsymbol{p}) &=&
\frac{1}{N}\sum_{j=1}^dp_j^{-1}x_jy_j - N
\nonumber \\
&=& \frac{1}{N}\sum_{j=1}^dp_j^{-1}(x_j-Np_j)(y_j-Np_j)
\nonumber \\
Q_2(\boldsymbol{x},\boldsymbol{y};N,\boldsymbol{p}) &=&
\frac{1}{2N(N-1)}\sum_{i,j=1}^dp_i^{-1}p_j^{-1}x_i(x_j-\delta_{ij})
y_i(y_j-\delta_{ij})\nonumber \\
&&~~~~-\frac{N-1}{N}\sum_{j=1}^dp_j^{-1}x_jy_j + {N\choose 2}
\label{three}
\end{eqnarray}
\end{rmk}
\begin{rmk}
If $d=2$
\[
Q_n(\boldsymbol{x},\boldsymbol{y};N,\boldsymbol{p}) = h_n(p_1)Q_n(x_1;N,p_1)Q_n(y_1;N,p_1),
\]
a product of the 1-dimensional Krawtchouk polynomials.
\end{rmk}
\begin{rmk}
\label{condense:rmk}
$Q_n(\boldsymbol{x},\boldsymbol{y};N,\boldsymbol{p})$ has the same form under grouping and adding disjoint collections of variables in $\boldsymbol{x}$, $\boldsymbol{y}$.

Let $A$ be a $d^\prime \times d$  $0-1$ matrix , $d^\prime < d$, with orthogonal rows and $\boldsymbol{x}^\prime = A\boldsymbol{x}$, $\boldsymbol{y}^\prime = A\boldsymbol{y}$. Then $\boldsymbol{X}^\prime$, $\boldsymbol{Y}^\prime$ are multinomial random vectors with parameters $N$, $\boldsymbol{p}^\prime = A\boldsymbol{p}$. In view of (\ref{Q_transform}) by setting variables in $\boldsymbol{s}$, $\boldsymbol{t}$ to be equal within groups defined by the mapping it is seen from the transform that
\begin{equation}
\mathbb{E}\big [Q_n(\boldsymbol{X},\boldsymbol{Y};N,\boldsymbol{p})\mid \boldsymbol{X}^\prime, \boldsymbol{Y}^\prime\big ]
= Q_n(\boldsymbol{X}^\prime,\boldsymbol{Y}^\prime;N,\boldsymbol{p}^\prime).
\label{condense:0}
\end{equation}
A particular case of (\ref{condense:0}) is taking $\boldsymbol{X}^\prime = (X_j)$, $\boldsymbol{Y}^\prime = (Y_j)$ and the other variables grouped with totals $N-X_j$ and $N-Y_j$. Then
\begin{eqnarray}
\mathbb{E}\big [Q_n(\boldsymbol{X},\boldsymbol{Y};N,\boldsymbol{p})\mid x_j, y_j\big ]
&=& Q_n(x_j,y_j;N,(p_j,N-p_j))
\nonumber \\
&=& 
 h_n(p_j)Q_n(x_j;N,p_j)Q_n(y_j;N,p_j),
 \label{twocase:0}
 \end{eqnarray}
 a product of 1-dimensional Krawtchouk polynomials.
 If $x_j=y_j=N$ then all the other variables are zero and (\ref{twocase:0})
 implies the identity
 \begin{eqnarray}
 Q_n(N\boldsymbol{e}_j, N\boldsymbol{e}_j;N,\boldsymbol{p})&=& 
 h_N(p_j)Q_n(N;N,p_j)Q_n(N;N,p_j)
 \nonumber \\
 &=& {N\choose n}(p_j^{-1}-1)^n.
 \label{twocase:00}
 \end{eqnarray}
\end{rmk}

\begin{rmk} Let ${\cal S}_d$ be the symmetric group of permutations on $1,2,\ldots,d$ and denote $\sigma(\boldsymbol{x}) = (x_{\sigma(1)},\ldots, x_{\sigma(d)})$
then \[Q_n(\sigma(\boldsymbol{x}),\sigma(\boldsymbol{y});N,\sigma(\boldsymbol{p}))=Q_n(\boldsymbol{x},\boldsymbol{y};N,\boldsymbol{p})\] is invariant under $\sigma \in {\cal S}_d$.
\end{rmk}
\begin{rmk}
There is an interesting probabilistic structure to the reproducing kernel polynomials.
Write
\begin{eqnarray}
&&Q_n(\boldsymbol{x},\boldsymbol{y};N,\boldsymbol{p})m(\boldsymbol{x};\boldsymbol{p})m(\boldsymbol{y};\boldsymbol{p}) 
\nonumber \\
&&=
\sum_{k=0}^n (-1)^{n-k}{N\choose k}
{N-k\choose n - k}m(\boldsymbol{x};\boldsymbol{p})m(\boldsymbol{y};\boldsymbol{p}) \zeta_{k} (\boldsymbol{x},\boldsymbol{y}),
\end{eqnarray}
where 
\[
m(\boldsymbol{x};\boldsymbol{p})m(\boldsymbol{y};\boldsymbol{p}) \zeta_{k} (\boldsymbol{x},\boldsymbol{y}) = \sum_{\boldsymbol{z};|\boldsymbol{z}|=k}
P(\boldsymbol{x}\mid \boldsymbol{z})P(\boldsymbol{y}\mid \boldsymbol{z})P(\boldsymbol{z})
\]
and $P(\boldsymbol{z}) = m(\boldsymbol{z};\boldsymbol{p})$, $P(\boldsymbol{x}\mid \boldsymbol{z}) = m(\boldsymbol{x}-\boldsymbol{z};\boldsymbol{p})$,
 $P(\boldsymbol{y}\mid \boldsymbol{z}) = m(\boldsymbol{y}-\boldsymbol{z};\boldsymbol{p})$.
The probabilistic structure of $m(\boldsymbol{x};\boldsymbol{p})m(\boldsymbol{y};\boldsymbol{p}) \zeta_{k}(\boldsymbol{x},\boldsymbol{y})$ is that $k$ observations are taken from a population and become the first duplicated entries in two samples $\boldsymbol{x}$ and $\boldsymbol{y}$. The remaining $|\boldsymbol{x}|-k$ and $|\boldsymbol{y}|-k$ observations in $\boldsymbol{x}$ and $\boldsymbol{y}$ are taken independently from the population. $m(\boldsymbol{x};\boldsymbol{p})m(\boldsymbol{y};\boldsymbol{p}) \zeta_{k}(\boldsymbol{x},\boldsymbol{y})$ is a bivariate multinomial distribution with $\boldsymbol{z}$ random elements  in common.
\end{rmk}
\begin{rmk}\label{center:0}
There is another form for the reproducing kernel polynomials where the terms are centered, which is useful for a chi-squared application in the next section. To ease notation we define $\boldsymbol{p}^{\boldsymbol{z}} = \prod_{j=1}^dp_j^{z_j}$, $\boldsymbol{z}! = \prod_{j=1}^dz_j!$,
$\boldsymbol{x}_{[\boldsymbol{z}]} = \prod_{j=1}^d{x_j}_{[z_j]}$, and centered terms
 $\boldsymbol{x}^c_{[\boldsymbol{z}]} = \boldsymbol{x}_{[\boldsymbol{z}]}  - N_{[|\boldsymbol{z}|]}\boldsymbol{p}^{\boldsymbol{z}}$. This is a natural centering because under a multinomial expectation
$\mathbb{E}\big[\boldsymbol{X}_{[\boldsymbol{z}]}\big ] = N_{[|\boldsymbol{z}|]}\boldsymbol{p}^{\boldsymbol{z}}$. We claim that for $1 \leq n \leq N$:
\begin{cor}
\begin{eqnarray}
Q_{n}(\boldsymbol{x},\boldsymbol{y};N,\boldsymbol{p})
&=& \sum_{\boldsymbol{z}\leq \boldsymbol{x},\boldsymbol{y}:1 \leq  |\boldsymbol{z}| \leq n}
{N\choose |\boldsymbol{z}|}{N-|\boldsymbol{z}|\choose n - |\boldsymbol{z}|}(-1)^{n-|\boldsymbol{z}|}
\nonumber \\
&&~~~~\times
{|\boldsymbol{z}|\choose \boldsymbol{z}}
\boldsymbol{x}^c_{[\boldsymbol{z}]}\boldsymbol{y}^c_{[\boldsymbol{z}]}\boldsymbol{p}^{-\boldsymbol{z}}N_{[|\boldsymbol{z}|]}^{-2}.
\label{claim:0}
\end{eqnarray}
\end{cor}
\begin{proof}
\begin{eqnarray}
\sum_{\boldsymbol{z}:|\boldsymbol{z}|\text{~fixed}}
{|\boldsymbol{z}|\choose \boldsymbol{z}}
\boldsymbol{x}^c_{[\boldsymbol{z}]}\boldsymbol{y^c}_{[\boldsymbol{z}]}\boldsymbol{p}^{-\boldsymbol{z}}N_{[|\boldsymbol{z}|]}^{-2}
&=& 
\sum_{\boldsymbol{z}:|\boldsymbol{z}|\text{~fixed}}
{|\boldsymbol{z}|\choose \boldsymbol{z}}
\boldsymbol{x}_{[\boldsymbol{z}]}\boldsymbol{y}_{[\boldsymbol{z}]}\boldsymbol{p}^{-\boldsymbol{z}}N_{[|\boldsymbol{z}|]}^{-2}
\nonumber \\
&&-\sum_{\boldsymbol{z}:|\boldsymbol{z}|\text{~fixed}}
{|\boldsymbol{z}|\choose \boldsymbol{z}}
\Big (\boldsymbol{x}_{[\boldsymbol{z}]}+\boldsymbol{y}_{[\boldsymbol{z}]}\Big )N_{[|\boldsymbol{z}|]}^{-1}
\nonumber \\
&&+\sum_{\boldsymbol{z}:|\boldsymbol{z}|\text{~fixed}}
{|\boldsymbol{z}|\choose \boldsymbol{z}}
\boldsymbol{p}^{\boldsymbol{z}}.
\label{temp:100}
\end{eqnarray}
The second term in (\ref{temp:100}) is
\begin{eqnarray*}
&&-\sum_{\boldsymbol{z}:|\boldsymbol{z}|\text{~fixed}}
{|\boldsymbol{z}|\choose \boldsymbol{z}}
\Big (\boldsymbol{x}_{[\boldsymbol{z}]}+\boldsymbol{y}_{[\boldsymbol{z}]}\Big )N_{[|\boldsymbol{z}|]}^{-1}
\nonumber \\
&&=-N_{[|\boldsymbol{z}|]}^{-1}|\boldsymbol{z}|!
\sum_{\boldsymbol{z}:|\boldsymbol{z}|\text{~fixed}}\Bigg (\prod_{j=1}^d{x_i\choose z_i}+\prod_{j=1}^d{y_i\choose z_i}\Bigg )
\nonumber \\
&&= -2N_{[|\boldsymbol{z}|]}^{-1}|\boldsymbol{z}|!{N \choose |\boldsymbol{z}|} = -2,
\end{eqnarray*}
the 3rd term is equal to 1, and the sum of the 2nd and 3rd terms is $-1$.
Applying the outer sum in (\ref{claim:0}) to the 2nd and 3rd terms in (\ref{temp:100}) 
\begin{eqnarray*}
-\sum_{1 \leq |\boldsymbol{z}| \leq n}
{N\choose |\boldsymbol{z}|}{N-|\boldsymbol{z}|\choose n - |\boldsymbol{z}|}(-1)^{n-|\boldsymbol{z}|} 
&=&
-\sum_{0 \leq |\boldsymbol{z}| \leq n}
{N\choose |\boldsymbol{z}|}{N-|\boldsymbol{z}|\choose N-n}(-1)^{n-|\boldsymbol{z}|} 
\nonumber \\
&&+ {N\choose n}(-1)^n
\nonumber \\
&=& 0 + {N\choose n}(-1)^n.
\end{eqnarray*}
The sum vanishes on the right side because ${N-|\boldsymbol{z}|\choose N-n}$ is a polynomial in $|\boldsymbol{z}|$ of degree $N-n$ and we have $n>0$ in what we are considering.
The two forms of $Q_n(\boldsymbol{x},\boldsymbol{y};N,p)$ (\ref{claim:0}) and (\ref{Kernel_poly}) then match up correctly.
\end{proof}
Another way of writing the reproducing kernel polynomials which is instructive is
\begin{eqnarray}
&&Q_{n}(\boldsymbol{x},\boldsymbol{y};N,\boldsymbol{p}) =
\nonumber \\
&& \sum_{1 \leq  |\boldsymbol{z}| \leq n}
{N\choose |\boldsymbol{z}|}{N-|\boldsymbol{z}|\choose n - |\boldsymbol{z}|}(-1)^{n-|\boldsymbol{z}|}
\nonumber \\
&&~~~~\times
\sum_{\boldsymbol{z}:|\boldsymbol{z}|\text{~fixed}}
{|\boldsymbol{z}|\choose \boldsymbol{z}}\boldsymbol{p}^{\boldsymbol{z}}
\Bigg (
\frac{
{\cal H}(\boldsymbol{z}\mid \boldsymbol{x})
}
{
{|\boldsymbol{z}|\choose \boldsymbol{z}}
\boldsymbol{p}^{\boldsymbol{z}}
}
-1\Bigg)\Bigg (
\frac{
{\cal H}(\boldsymbol{z}\mid \boldsymbol{y})
}
{
{|\boldsymbol{z}|\choose \boldsymbol{z}}
\boldsymbol{p}^{\boldsymbol{z}}}-1\Bigg),
\label{claim:100}
\end{eqnarray}
\end{rmk}
where
\[
{\cal H}(\boldsymbol{z}\mid \boldsymbol{x}) = \prod_{j=1}^d {x_j\choose z_j}\Big /{N \choose |\boldsymbol{z}|} 
\]
is the hypergeometric probability of obtaining a sub-sample configuration of $\boldsymbol{z}$ from $\boldsymbol{x}$. The expression (\ref{claim:100}) follows in an easy way from (\ref{claim:0}) by noting that
\[
{\cal H}(\boldsymbol{z}\mid \boldsymbol{x}) = \Big (\boldsymbol{x}_{[\boldsymbol{z}]}/\boldsymbol{z}!\Big )\Big / {N \choose |\boldsymbol{z}|}
\]
and simplifying.
\subsection{A statistical application of kernel polynomials}\label{s:chi}
Recall the classical chi-squared goodness of fit test. Let ${\cal X}$ be a finite set, $p(x) > 0$, $\sum_{x \in \cal{X}} p(x) = 1$, a probability distribution on ${\cal X}$. Let $X_1,X_2,\ldots,X_r$ be ${\cal X}$ valued random variables. To test if the $\{X_i\}_{i=1}^r$ are $p(x)$ distributed one computes
\[
X^2 = \sum_x\frac{\big (N_x - rp(x)\big )^2}{rp(x)},\>\text{with~} N_x = \#\{i:X_i=x\}.
\]
A common problem is that for large sample sizes ($r$ large) usually the test rejects the null hypothesis and one doesn't know what is causing the rejection. One classical solution to this problem is to decompose the chi-squared statistic into \emph{components}.
Let $\{u^{(l)}(x),\>0 \leq l \leq |{\cal X}|-1\}$ be orthonormal functions on ${\cal X}$ with respect to $p(x)$. Let $\widehat{p}(x)=N_x/r$ be the empirical distribution of the data, and define 
\[
\widetilde{p}(l) = \sum_xu^{(l)}(x)\widehat{p}(x),\text{~the~}l^{\text{th}}\text{~transform}.
\]
Then
\[
X^2 = r\sum_{l=1}^{|{\cal X}|-1}|\widetilde{p}(l)|^2.
\]
If the null hypothesis is true, asymptotically $\{r|\widetilde{p}(x)|^2\}$ are independent with chi-squared distributions having 1 degree of freedom, thus resolving the original $X^2$ statistic.

An extensive development of this approach is in \citet{S2007}, who gives history and examples. Her main idea is to use the eigenvectors of natural reversible Markov chains on ${\cal X}$ having $p(x)$ as stationary distribution.

The multivariate Krawtchouk polynomials can be used in this way where $\boldsymbol{X}_1,\boldsymbol{X}_2,\ldots,\boldsymbol{X}_r$ take values in the configuration space of $N$ balls dropped into $d$ boxes and $p(\boldsymbol{x}) = m(\boldsymbol{x};N,\boldsymbol{p})$, the multinomial distribution over. In this case $|{\cal X}| = {N+d-1\choose d-1}$ and it is natural to break the components into linear, quadratic, cubic, $...$ pieces. The following considerations show how the kernel polynomials can be used for this task.
\begin{thm}
Let ${\cal X}$ be the configuration space of $N$ balls dropped into $d$ boxes. Let
 \[
Q^\circ_{\boldsymbol{n}}(\boldsymbol{x};\boldsymbol{u})
= {N \choose \boldsymbol{n}, N - |\boldsymbol{n}|}^{-1/2}Q_{\boldsymbol{n}}(\boldsymbol{x};\boldsymbol{u})
\]
 be the orthonormal multivariate Krawtchouk polynomials based on the orthonormal basis $u^{(l)}$ as in Section \ref{OPsection}. Let $\widehat{p}(\boldsymbol{x})$ be the empirical measure of $\boldsymbol{X}_1,\ldots,\boldsymbol{X}_r$, a sample of size $r$ from ${\cal X}$, and
$\widetilde{p}(\boldsymbol{\boldsymbol{n}}) =\sum_{\boldsymbol{x}}Q^\circ_{\boldsymbol{n}}(\boldsymbol{x};\boldsymbol{u})\widehat{p}(\boldsymbol{x})$.
Finally define, for $1 \leq i \leq N$, $\widetilde{p}(i)^2 = \sum_{\boldsymbol{n}:|\boldsymbol{n}| = i}|\widetilde{p}(\boldsymbol{n})|^2$.

\noindent
Then
\begin{equation}
\widetilde{p}(i)^2 =
\sum_{\boldsymbol{x},\boldsymbol{y}}Q_i(\boldsymbol{x},\boldsymbol{y};N,\boldsymbol{p})
\widehat{p}(\boldsymbol{x})\widehat{p}(\boldsymbol{y}),
\label{chi:100}
\end{equation}
and
\[
\sum_{i=1}^N r\widetilde{p}(i)^2 =  X^2(Nr),
\]
where $X^2(Nr)$ is the chi-squared statistic based on dropping the $Nr$ balls into $d$ urns.
A particular case from (\ref{chi:100}) is
\begin{equation}
r\widetilde{p}(1)^2 =   r\sum_{j=1}^d\frac{(\bar{x}_j-Np_j)^2}{Np_j},
\label{chi:200}
\end{equation}
a goodness of fit statistic for testing whether the proportions are correct.

Under the null hypothesis
$r\widetilde{p}(i)^2$ are asymptotically independent chi-squared components with ${i+d-2\choose d-2}$ degrees of freedom, $i=1,2,\ldots,N$. 
\end{thm}
\begin{proof}
\[
Q_i(\boldsymbol{x},\boldsymbol{y};N,\boldsymbol{p}) = \sum_{|\boldsymbol{n}|=i}
Q^\circ_{\boldsymbol{n}}(\boldsymbol{x};\boldsymbol{u})Q^\circ_{\boldsymbol{n}}(\boldsymbol{y};\boldsymbol{u})
\]
 for any orthonormal basis $\boldsymbol{u}$. Therefore
\begin{eqnarray*}
\sum_{\boldsymbol{x},\boldsymbol{y}}Q_i(\boldsymbol{x},\boldsymbol{y};N,\boldsymbol{p})
\widehat{p}(\boldsymbol{x})\widehat{p}(\boldsymbol{y})
&=& 
 \sum_{|\boldsymbol{n}|=i}
\sum_{\boldsymbol{x}}Q^\circ_{\boldsymbol{n}}(\boldsymbol{x};\boldsymbol{u})\widehat{p}(\boldsymbol{x})\sum_{\boldsymbol{y}}Q^\circ_{\boldsymbol{n}}(\boldsymbol{y};\boldsymbol{u})\widehat{p}(\boldsymbol{y})
\nonumber \\
&=&
\sum_{\boldsymbol{n}:|\boldsymbol{n}| = i}|\widetilde{p}(\boldsymbol{n})|^2
\nonumber \\
&=&
\widetilde{p}(i)^2.
\end{eqnarray*}
$Q^\circ_{\boldsymbol{n}}(\boldsymbol{x};\boldsymbol{u})$ is indexed by a $d-1$ dimensional vector $\boldsymbol{n}$ with $|\boldsymbol{n}|=i$. The number of these orthonormal polynomials is the number of partitions of $i$ into $d-1$ parts, ${i+d-2\choose d-2}$, which are the degrees of freedom of the associated chi-squared. Note that the chi-squared degrees of freedom in the partition add up correctly because
\[
\sum_{i=1}^N {i+d-2\choose d-2} = {N+d-1\choose d-1} -1,
\]
the degrees of freedom of a full multinomial chi-squared goodness of fit.

From the explicit form of $Q_1(\boldsymbol{x},\boldsymbol{y};N,\boldsymbol{p})$ in (\ref{three}), Remark \ref{polys:300}
\begin{eqnarray*}
\widetilde{p}(1)^2 &=&
\sum_{\boldsymbol{x},\boldsymbol{y}}\sum_{j=1}^d\frac{(x_j-Np_j)(y_j-Np_j)}{Np_j}\widehat{p}(\boldsymbol{x})\widehat{p}(\boldsymbol{y})
\nonumber \\
&=& \sum_{j=1}^d\frac{(\bar{x}_j-Np_j)^2}{Np_j},
\end{eqnarray*}
showing that (\ref{chi:200}) is correct.
\end{proof}
\begin{rmk}
Note that the $\widetilde{p}(i)^2$ do not depend on the basis $u^{(l)}$. Formula (\ref{chi:100}) is useful, for example, when $d$ is large and $r$ is moderate (say a few thousand). Then the formulae of Section \ref{OPsection}, Remark \ref{polys:300} can be summed over in this example. The various $\widetilde{p}(i)^2$ can be combined using the Poisson kernel of Theorem \ref{Thm1} over the sample values, with $0 < \rho < 1$ fixed.
For an example (testing if the zeros of the zeta function fit random matrix theory) see \citet{CD2003}. See also \citet{S2017}.
\end{rmk}
\begin{rmk} If $\boldsymbol{p}$ is estimated by $\widehat{\boldsymbol{p}}=\bar{\boldsymbol{x}}/N$ and substituted in the total chi-squared goodness of fit statistic then the degrees of freedom are 
${N+d-1\choose d-1} -1 -(d-1)$. The test statistic is then
$\sum_{i=2}^{N-1}\widetilde{p}_{\widehat{\boldsymbol{p}}}(i)^2$ where $\boldsymbol{p}$ is replaced by $\widehat{\boldsymbol{p}}$.
\end{rmk}
\begin{rmk}
Another way of expressing $\widetilde{p}^2(i)$ from (\ref{claim:0}), using the compact notation in Remark \ref{center:0} for $\boldsymbol{p}^{\boldsymbol{z}}=\prod_{j=1}^dp_j^{z_j}$, is that
\begin{eqnarray*}
\widetilde{p}^2(i) &=&
\sum_{1 \leq  |\boldsymbol{z}| \leq i}
{N\choose |\boldsymbol{z}|}{N-|\boldsymbol{z}|\choose i - |\boldsymbol{z}|}(-1)^{i-|\boldsymbol{z}|}
\nonumber \\
&&~~~~\times
\sum_{\boldsymbol{z}:|\boldsymbol{z}|\text{~fixed}}
{|\boldsymbol{z}|\choose \boldsymbol{z}}\boldsymbol{p}^{\boldsymbol{z}}
\Bigg (
\frac{
{\overline{\cal H}(\boldsymbol{z}\mid \boldsymbol{x}_1,\ldots,\boldsymbol{x}_r)}
}
{
{|\boldsymbol{z}|\choose \boldsymbol{z}}
\boldsymbol{p}^{\boldsymbol{z}}
}
-1\Bigg)^2,
\end{eqnarray*}
where 
\begin{eqnarray*}
\overline{\cal H}(\boldsymbol{z}\mid \boldsymbol{x}_1,\ldots,\boldsymbol{x}_r)
&= &\frac{1}{r}\sum_{j=1}^r{\cal H}(\boldsymbol{z}\mid \boldsymbol{x}_j)
\nonumber \\
&=& \sum_{\boldsymbol{x}}{\cal H}(\boldsymbol{z}\mid \boldsymbol{x})\widehat{p}(\boldsymbol{x})
\end{eqnarray*}
is the empirical probability of a configuration $\boldsymbol{z}$ in  a sub-sample of size $|\boldsymbol{z}|$ from the pooled $\boldsymbol{x}_j$, $j=1,\ldots, N$.  
Note that
\[
\mathbb{E}\big [\overline{\cal H}(\boldsymbol{z}\mid \boldsymbol{X}_1,\ldots,\boldsymbol{X}_r)\big ]
= {|\boldsymbol{z}| \choose \boldsymbol{z}}\boldsymbol{p}^{\boldsymbol{z}},
\]
since a sub-sample, unconditional on $\boldsymbol{X}_1,\ldots, \boldsymbol{X}_r$, has a multinomial distribution.
The $i$th chi-squared component $r\widetilde{p}^2(i)$ is therefore testing whether the empirical $i$-sub-sampling probabilities from the data $\boldsymbol{x}_1,\ldots,\boldsymbol{x}_r$ are consistent with the null multinomial distribution, taking into account that the lower order sub-samples are consistent. 

\end{rmk}

\subsection{Reproducing kernel polynomials on the product Poisson distribution}\label{s:Poisson}
The reproducing kernel polynomials $\{Q_n^P(\boldsymbol{x},\boldsymbol{y};\boldsymbol{\mu})\}_{n=0}^\infty$ on the product Poisson distribution 
\begin{equation}
P(\boldsymbol{x};\boldsymbol{\mu})=\prod_{i=1}^de^{-\mu_i}\frac{\mu_i^{x_i}}{x_i!},\>\boldsymbol{x}\in \mathbb{Z}_+^d
\label{pp:0}
\end{equation}
 are now obtained as a limit from the reproducing kernel polynomials on the multinomial. They could also be obtained from the product set of Poisson-Charlier polynomials.
\begin{thm}

\noindent
\emph{(a)}
Let $\boldsymbol{X}^{(N)}$ be a $d+1$ multinomial $(N,(\boldsymbol{p},p_{d+1}))$ random vector with $\boldsymbol{p}= (p_1,\ldots,p_d)$ and $p_{d+1}= 1 - |\boldsymbol{p}|$. Then as
$N\to \infty$, $\boldsymbol{p} \to 0$, with $N\boldsymbol{p} \to \boldsymbol{\mu}$ the first $d$ elements of $\boldsymbol{X}^{(N)}$ have a limit Poisson distribution (\ref{pp:0}) and
\begin{equation}
Q_n\big (\boldsymbol{x}^{(N)},\boldsymbol{y}^{(N)};N,(\boldsymbol{p},p_{d+1})\big ) \to Q_n^P(\boldsymbol{x},\boldsymbol{y};\boldsymbol{\mu}).
\label{pp:1}
\end{equation}

\noindent
\emph{(b)}
The Poisson kernel, non-negative for $0\leq \rho \leq 1$, is
\begin{eqnarray}
&&1 + \sum_{n=1}^\infty \rho^nQ_n^P(\boldsymbol{x},\boldsymbol{y};\boldsymbol{\mu})
\nonumber \\
&&=e^{|\boldsymbol{\mu}|\rho}\sum_{\boldsymbol{z}\leq \boldsymbol{x},\boldsymbol{y}}
\rho^{|\boldsymbol{z}|}(1-\rho)^{|\boldsymbol{x}|+ |\boldsymbol{y}| -2|\boldsymbol{z}|}
\prod_{i=1}^d\frac{{x_i}_{[z_i]}{y_i}_{[z_i]}}{\mu_i^{z_i}z_i!}.
\label{pp:2}
\end{eqnarray}

\noindent
\emph{(c)}
An explicit expression for the reproducing kernel polynomials is
\begin{equation}
Q_n^P(\boldsymbol{x},\boldsymbol{y};\boldsymbol{\mu})=
\sum_{\boldsymbol{z}\leq n,\boldsymbol{x},\boldsymbol{y}}
\frac{|\boldsymbol{\mu}|^{n-|\boldsymbol{z}|}}{(n-|\boldsymbol{z}|)!}
C_{n-|\boldsymbol{z}|}(|\boldsymbol{x}|+|\boldsymbol{y}|-2|\boldsymbol{z}|;|\boldsymbol{\mu}|)
\prod_{i=1}^d\frac{{x_i}_{[z_i]}{y_i}_{[z_i]}}{\mu_i^{z_i}z_i!},
\label{pp:3}
\end{equation}
where $\{C_n(x;\lambda)\}_{n=0}^\infty$ are the Poisson-Charlier polynomials,
with generating function
\[
\sum_{n=0}^\infty C_n(x;\lambda)\frac{z^n}{n!} =
e^z\Big (1 - \frac{z}{\lambda}\Big )^x.
\]
\end{thm}
\begin{proof}

\noindent
\emph{(a)}
 The convergence in distribution of the multinomial to the Poisson is well known. Now consider the \emph{pgf} of the first $d$ elements in the $d+1$ dimensional vectors $(\boldsymbol{X}^{(N)},\boldsymbol{Y}^{(N)})$. Setting $s_{d+1}=t_{d+1} = 1$, $p_{d+1}= 1 - \sum_{i=1}^dp_i$, in a $d+1$ dimensional version of (\ref{pgf:0})
\begin{eqnarray}
&&\mathbb{E}\big [\prod_{i,j=1}^ds_i^{X^{(N)}_i}t_j^{Y^{(N)}_j}\big ]
=
\Big ((1-\rho)\big (1+\sum_{i=1}^dp_i(s_i-1)\big )\big (1+ \sum_{j=1}^dp_j(t_j-1)\big ) 
\nonumber \\
&&~~~~~~~~~~~~~~~~~~~~~~~~~~~~~~~~~~~
+ \rho (1 + \sum_{i=1}^dp_i(s_it_i-1)\Big )^N.
\label{pgf:00}
\end{eqnarray}
The limit expression of (\ref{pgf:00})  as
$N\to \infty$, $\boldsymbol{p} \to 0$, with $N\boldsymbol{p} \to \boldsymbol{\mu}$, is
\begin{eqnarray}
&&\mathbb{E}\big [\prod_{i,j=1}^ds_i^{X_i}t_j^{Y_j}\big ]
\nonumber \\
&&=\exp \Big \{(1-\rho)\Big (\sum_{i=1}^d\mu_i(s_i-1) + \sum_{i=1}^d\mu_i(t_i-1)\Big )
+ \rho\sum_{i=1}^d\mu_i(s_it_i-1) \Big \}.~~~~~~~~
\label{pgf:01}
\end{eqnarray}

\noindent
\emph{(b),(c)}
The Poisson kernel is the coefficient of $\prod_{i=1}^ds_i^{x_i}t_i^{y_i}$ in 
(\ref{pgf:01}), divided by $P(\boldsymbol{x};\boldsymbol{\mu})P(\boldsymbol{y};\boldsymbol{\mu})$, which is equal to (\ref{pp:2}). The explicit expression (\ref{pp:3}) follows
immediately as the coefficient of $\rho^n$ in (\ref{pp:2}).

\end{proof}

\subsection{Duplication formula for the reproducing kernel polynomials}
\label{duplication:10}
Define
\begin{equation}
K(\boldsymbol{x},\boldsymbol{y},z) = \sum_{n=0}^NQ_{n}(z;N,p)Q_{n}(\boldsymbol{x},\boldsymbol{y};N,\boldsymbol{p}),
\label{KK:0}
\end{equation}
where $\big \{Q_n(z;N,p)\big \}$ are the 1-dimensional Krawtchouk polynomials.
Our interest is in finding parameter values $p,\boldsymbol{p}$ such that $K(\boldsymbol{x},\boldsymbol{y},z) \geq 0$, leading to a duplication formula for the Kernel polynomials extending the \citet{E1969} formulae (\ref{duplicate:0}) and (\ref{tripple:0}). Note that $p$ (with $q = 1-p$) is an independent parameter not depending on $\boldsymbol{p}$.
\bigskip

\noindent
\begin{thm} \label{Duplicationthm}

\noindent
\emph{(a)}
$K(\boldsymbol{x},\boldsymbol{y},z) \geq 0$ for $\boldsymbol{x},\boldsymbol{y}$ in the support of the multinomial $(N,\boldsymbol{p})$ distribution and $z=0,1,\ldots, N$ if and only if 
\begin{equation}
1-p \leq \min_{j\in [d]}p_j.
\label{condition:0}
\end{equation}

\noindent
\emph{(b)}
If (\ref{condition:0}) holds there is a duplication formula
\begin{equation}
Q_n(\boldsymbol{x},\boldsymbol{y};N,\boldsymbol{p}) = h_n(p)\mathbb{E}_{\varphi_{\boldsymbol{x},\boldsymbol{y}}}\bigl [Q_n(Z;N,p)\bigr ],
\label{kduplicate:0}
\end{equation}
with $h_n(p) = {N\choose n}(p/q)^n$, where $Z$ has a probability distribution 
\begin{equation}
\varphi_{\boldsymbol{x},\boldsymbol{y}}(z) = {N\choose z}p^zq^{N-z}K(\boldsymbol{x},\boldsymbol{y},z),\>
z=0,1,\ldots N.
\label{varphi:0}
\end{equation}
\end{thm}

\begin{proof}
The transform of $K(\boldsymbol{x},\boldsymbol{y},z)$ over $\boldsymbol{x},\boldsymbol{y},z$
can be found from the transforms (\ref{KTR:0}) 
and (\ref{Q_transform}).
Taking expectation with $\boldsymbol{X},\boldsymbol{Y}, Z$ independent 
\begin{eqnarray}
&&\mathbb{E}\Big [\prod_{i=1}^ds_i^{X_i}t_i^{Y_i}\psi^ZK(\boldsymbol{X},\boldsymbol{Y}, Z)  \Big ]
\nonumber \\
&&=\sum_{n=0}^N{N\choose n}
\Big (T_0(\boldsymbol{s})T_0(\boldsymbol{t})\Big )^{N-n}
\Big (\sum_{j=1}^dp_js_jt_j - T_0(\boldsymbol{s})T_0(\boldsymbol{t})
\Big )^{n}
\nonumber \\
&&~~~~~~~~~~~~~~~~~~~~~~~~~~~~~~~~~~\times \big (q(1-\psi)\big )^n\big (p\psi+q\big )^{N-n}
\nonumber \\
&&=
\Bigl [ q(1-\psi)\Big (\sum_{j=1}^dp_js_jt_j - T_0(\boldsymbol{s})T_0(\boldsymbol{t})\Big )
+ (p\psi+q)T_0(\boldsymbol{s})T_0(\boldsymbol{t})\Bigr ]^N
\nonumber \\
&&= \Bigl [q\sum_{j=1}^dp_js_jt_j + \psi\Big (T_0(\boldsymbol{s})T_0(\boldsymbol{t}) - q\sum_{j=1}^dp_js_jt_j\Big )\Bigr ]^N.
\label{jointt:0}
\end{eqnarray}
The \emph{pgf} of the distribution in $(\boldsymbol{X},\boldsymbol{Y})$
\[
m(\boldsymbol{x},\boldsymbol{p})
m(\boldsymbol{y},\boldsymbol{p})
K(\boldsymbol{x},\boldsymbol{y},z)
\]
for $z = 0,\ldots N$ is the coefficient of 
$\psi^z$
divided by ${N\choose z}p^zq^{N-z}$ in (\ref{jointt:0}), which is equal to
\begin{equation}
p^{-z}\Big [T_0(\boldsymbol{s})T_0(\boldsymbol{t}) - q \sum_{i=1}^dp_is_it_i\Big ]^z
\Big [\sum_{i=1}^dp_is_it_i\Big ]^{N-z}.
\label{jointt:1}
\end{equation}
The coefficients of the off-diagonal terms $s_it_j$ are non-negative and the coefficients of the diagonal terms $s_jt_j$ are $p_j^2 - qp_j$ in the first term in (\ref{jointt:1}), which are non-negative if and only if for $j\in [d]$, $q \leq p_j$, equivalent to $q \leq \min_{j\in [d]}p_j$ or $p \geq \max_{j\in [d]}q_j$. The duplication formula (\ref{kduplicate:0}) follows easily.
\end{proof}
\begin{rmk}
If (\ref{condition:0}) holds then $p \geq \frac{1}{2}$ and
$|Q_n(Z;N,p)| \leq 1$ in (\ref{kduplicate:0}). Therefore there is an inequality that
\begin{equation}
|Q_n(\boldsymbol{x},\boldsymbol{y};N,\boldsymbol{p})| \leq h_n(p) = {N\choose n}(p/q)^n.
\label{twocase:1}
\end{equation}
A tight bound is
\begin{equation}
|Q_n(\boldsymbol{x},\boldsymbol{y};N,\boldsymbol{p})| \leq {N\choose n} \big((\min_{j\in [d]}p_j)^{-1}-1\big )^n
\label{twocase:2}
\end{equation}
attained when $\boldsymbol{x}=\boldsymbol{y}= N\boldsymbol{e}_{j^*}$, where $j^*$ is the index where $p_{j^*}$ is minimal. The bound (\ref{twocase:2}) is found by taking 
$p=1-\min_{j\in [d]}p_j$ in (\ref{twocase:1}) and tightness follows from (\ref{twocase:00}).
\end{rmk}
\begin{rmk}
It is straightforward to derive an explicit formula for $\varphi_{\boldsymbol{x}\boldsymbol{y}}(\chi)$. The transform of this density with respect to $\chi$ is
\begin{eqnarray}
&&\sum_{n=0}^N\big (q(1-\psi)\big )^n\big (p\psi+q)^{N-n}
Q_n(\boldsymbol{x},\boldsymbol{y};N,\boldsymbol{p})
\nonumber \\
&&=m(\boldsymbol{x};N,\boldsymbol{p})^{-1}m(\boldsymbol{y};N,\boldsymbol{p})^{-1}
\sum_{\boldsymbol{z}}(1-\psi)^{|\boldsymbol{z}|}\psi^{N-|\boldsymbol{z}|}q^{|\boldsymbol{z}|}
\nonumber \\
&&~~\times
{N\choose\boldsymbol{z},N-|\boldsymbol{z}|}
\prod_{i=1}^dp_i^{z_i}
{N-|\boldsymbol{z}|\choose \boldsymbol{x}-\boldsymbol{z}}\prod_{i=1}^dp_i^{x_i-z_i}
{N-|\boldsymbol{z}|\choose \boldsymbol{y}-\boldsymbol{z}}\prod_{i=1}^dp_i^{y_i-z_i},~~~
\label{psi:0}
\end{eqnarray}
from (\ref{KTR:0}) and (\ref{RKPoisson}). Inverting (\ref{psi:0}) 
\begin{eqnarray}
&&\varphi_{\boldsymbol{x}\boldsymbol{y}}(\chi) = 
m(\boldsymbol{x};N,\boldsymbol{p})^{-1}m(\boldsymbol{y};N,\boldsymbol{p})^{-1}
\sum_{\boldsymbol{z}}(-1)^{N - |\boldsymbol{z}| - \chi}{|\boldsymbol{z}| \choose N - \chi}q^{|\boldsymbol{z}|}
\nonumber \\
&&~~\times
{N\choose\boldsymbol{z},N-|\boldsymbol{z}|}
\prod_{i=1}^dp_i^{z_i}
{N-|\boldsymbol{z}|\choose \boldsymbol{x}-\boldsymbol{z}}\prod_{i=1}^dp_i^{x_i-z_i}
{N-|\boldsymbol{z}|\choose \boldsymbol{y}-\boldsymbol{z}}\prod_{i=1}^dp_i^{y_i-z_i}.~~~
\label{psi:1}
\end{eqnarray}
\end{rmk}
\begin{cor}\label{ThreeK}
The 1-dimensional Krawtchouk polynomial triple product sum
\begin{equation}
K(x,y,z;N,r,s)= \sum_{n=0}^Nh_n(s)Q_n(x;N,s)Q_n(y;N,s)Q_n(z;N,r) \geq 0
\label{xyzrs:0}
\end{equation}
if and only if $1-r \leq \min(s,1-s)$ with $r,s\in [0,1]$.
Under these conditions
\[
\nu_{xy} = {N\choose z}r^z(1-r)^{N-z}K(x,y,z;N,r,s)
\]
is a probability distribution in $z=0,1,\ldots ,N$ and there is a duplication formula
\begin{equation}
h_n(s)Q_n(x;N,s)Q_n(y;N,s) = h_n(r)\mathbb{E}_{\nu_{xy}}\bigl [Q_n(Z;N,r)\bigr ].
\label{kduplicate:1}
\end{equation}
\end{cor}

\begin{proof} In Theorem \ref{Duplicationthm} take $d=2$, $p_1 = s,p_2 = 1-s$, and $x=x_1,y=y_1$.
Then $Q_{n}(\boldsymbol{x},\boldsymbol{y};N,\boldsymbol{p}) =  h_n(s)Q_n(x;N,s)Q_n(y;N,s)$.
The sum (\ref{xyzrs:0}) is non-negative if and only if 
$1-r \leq \min(s,1-s)$.
\end{proof}
The non-negative sum (\ref{xyzrs:0}) is also in \citet{DG2012} with a different proof.
If $r=s$ (\ref{kduplicate:1}) is Eagleson's formula.

The measure $\varphi_{\boldsymbol{x},\boldsymbol{y}}$ has an interesting probabilistic interpretation in terms of matching in two multinomial sequences of trials.
\begin{thm} \label{matchthm}
Let $\{\xi_j\}_{j=1}^N$, $\{\eta_j\}_{j=1}^N$ be two independent multinomial $(N,\boldsymbol{p})$ sequences of trials such that for $j=1,\ldots, N$
\[
P(\xi_j = l) = p_l,\>P(\eta_j = l) = p_l,\>l=1,2,\ldots, d.
\]
Denote the multinomial count vectors from the trials as $\boldsymbol{X}$, $\boldsymbol{Y}$.
Denote $M$ as the set of matched pairs in the two series of multinomial trials.
That is 
\[
M= \big \{(\xi_j,\eta_j):\xi_j=\eta_j,\>j=1,\ldots ,N\big \}.
\]
Let $\{B_{jk}\}_{1 \leq j \leq N; 1 \leq k \leq d}$ be an independent array of Bernoulli trials such that for each $j$, $P(B_{jk}=1) = \tau_k := q/p_k$, choosing $q \leq \min_{k=1}^dp_k$. Thin the set $M$ to $R$ randomly by the rule
\[
R = \{(\xi_j,\eta_j): (\xi_j,\eta_j) \in M \text{~and~}B_{j\xi_j}=1,j=1,\ldots,N\}.
\]
Let $Z$ be a random variable with measure $\varphi_{\boldsymbol{x},\boldsymbol{y}}$.
Then $N-Z$ is distributed as the number of elements in the thinned match set $R$, conditional on $\boldsymbol{X}=\boldsymbol{x}$, $\boldsymbol{Y}= \boldsymbol{y}$.
\end{thm}
\begin{proof}
We show that the pgf of $N-Z$ agrees with the pgf of the number of elements of $R$.
The pgf of $N-Z$ is
\begin{eqnarray}
&&\sum_{n=0}^N\mathbb{E}\big[\psi^{N-Z}Q_{n}(Z;N,p)\big ]Q_{n}(\boldsymbol{x},\boldsymbol{y};N,\boldsymbol{p})
\nonumber \\
&&~= \sum_{n=0}^N\big (q(\psi-1)\big )^n\big (p+q\psi\big )^{N-n}Q_{n}(\boldsymbol{x},\boldsymbol{y};N,\boldsymbol{p})
\nonumber \\
&&~=\sum_{r=0}^N
{N\choose r}(-1)^{r}(1-\psi)^{r}
\nonumber \\
&&~~~
\times 
\sum_{\boldsymbol{z}: |\boldsymbol{z}|=r,\boldsymbol{z}\leq \boldsymbol{x},\boldsymbol{y}}
\frac{ q^r\cdot
{r\choose \boldsymbol{z}}\prod_{i=1}^dp_i^{z_i} 
\cdot{N-r\choose \boldsymbol{x}-\boldsymbol{z}}\prod_{i=1}^dp_i^{x_i-z_i}
\cdot{N-r\choose \boldsymbol{y}-\boldsymbol{z}}\prod_{i=1}^dp_i^{y_i-z_i}
}
{
m(\boldsymbol{x};N,\boldsymbol{p}) m(\boldsymbol{y};N,\boldsymbol{p})
}
\nonumber \\
\label{pgdNmz}
\end{eqnarray} 
Expectation in the first line of (\ref{pgdNmz}) is with respect to a Binomial $(N,p)$ distribution; the second line follows from the transform (\ref{KTR:0}) and the third line follows from the Poisson Kernel (\ref{RKPoisson}).

The pgf of the number of elements in $R$ is now calculated using an inclusion-exclusion argument.
Let $I_j$ be the indicator function that $(\xi_j,\eta_j) \in R$, $j=1,\ldots,N$, conditional on the configuration 
 $\boldsymbol{X}=\boldsymbol{x}, \boldsymbol{Y}= \boldsymbol{y}$. $\{I_j\}_{j=1}^N$ are an exchangeable set of random variables. Then by probabilistic reasoning, considering the configuration $\boldsymbol{z}$ in trials $1,\ldots,r$ for which matches occur,
\begin{eqnarray}
&&\mathbb{E}\big [I_1\cdots I_r\big ]
\nonumber \\
&&=
\sum_{\boldsymbol{z}: |\boldsymbol{z}|=r,\boldsymbol{z}\leq \boldsymbol{x},\boldsymbol{y}}
\frac{
\prod_{i=1}^d\tau_i^{z_i}\cdot{r\choose \boldsymbol{z}}\prod_{i=1}^d{p_i}^{2z_i} 
\cdot{N-r\choose \boldsymbol{x}-\boldsymbol{z}}\prod_{i=1}^dp_i^{x_i-z_i}
\cdot{N-r\choose \boldsymbol{y}-\boldsymbol{z}}\prod_{i=1}^dp_i^{y_i-z_i}
}
{
m(\boldsymbol{x};N,\boldsymbol{p}) m(\boldsymbol{y};N,\boldsymbol{p})
}
\nonumber \\
&&=
\sum_{\boldsymbol{z}: |\boldsymbol{z}|=r,\boldsymbol{z}\leq \boldsymbol{x},\boldsymbol{y}}
\frac{
q^r\cdot{r\choose \boldsymbol{z}}\prod_{i=1}^d{p_i}^{z_i} 
\cdot{N-r\choose \boldsymbol{x}-\boldsymbol{z}}\prod_{i=1}^dp_i^{x_i-z_i}
\cdot{N-r\choose \boldsymbol{y}-\boldsymbol{z}}\prod_{i=1}^dp_i^{y_i-z_i}
}
{
m(\boldsymbol{x};N,\boldsymbol{p}) m(\boldsymbol{y};N,\boldsymbol{p})
}.
\label{in:out:1}
\end{eqnarray}
The pgf of $R$, from a very classical calculation, is
\begin{equation}
\mathbb{E}\big [\prod_{j=1}^N\big (\psi I_j + 1 - I_j\big )\big ]
=
\sum_{r=0}^N{N\choose r}(-1)^r(1-\psi)^r\mathbb{E}\big [I_1\cdots I_r\big ].
\label{Z:3}
\end{equation}
The two pgfs (\ref{pgdNmz}) and (\ref{Z:3}) are identical so $N-Z$ has the same distribution as the number of entries in $R$.
\end{proof}
\begin{cor}
If the multinomial trial outcomes are equally likely with $p_j=1/d$, $j=1,\ldots,d$ then it is possible to choose $q=1/d$ in Theorem \ref{matchthm} implying that $\tau_j=1$, $j=1,\ldots,d$. Then there is no thinning of the set of matched pairs $M$ and $N-Z$ is distributed as the distribution of the number of matching pairs conditional on  $\boldsymbol{X}=\boldsymbol{x}, \boldsymbol{Y}= \boldsymbol{y}$.
\end{cor}
\begin{cor}
If $d=2$ and $p_1 \geq 1/2$, take $q=p_2 = \min\{p_1,p_2\}$. Then 
$\varphi_{\boldsymbol{x},\boldsymbol{y}}$ is the mixing measure in Eagleson's hypergroup formula. Theorem \ref{matchthm} gives a new interpretation of this measure in terms of matching in two sets of $N$ binomial trials.
\end{cor}
The moments of $Z$ can in principle be found from (\ref{kduplicate:0}). 
The $r$th moment is a polynomial of degree $r$ in $\boldsymbol{x},\boldsymbol{y}$.
In particular
\[
\mathbb{E}\big [Z\big ] = Np - qQ_1(\boldsymbol{x},\boldsymbol{y};N,\boldsymbol{p}).
\]
The $r$th falling factorial moment 
\[
\mathbb{E}\big [(N-Z)_{[r]}\big ]={N\choose r}\mathbb{E}\big[I_1\cdots I_r\big].
\]
$\mathbb{E}\big[I_1\cdots I_r\big]$ is the probability that a particular $r$ trials belong to the match set $R$, given by (\ref{in:out:1}).
\section{Bivariate distributions and reproducing kernel polynomials}\label{section:4}
This section applies the theorems above to bivariate distributions and Markov chains. 
We first characterize a class of bivariate multinomial distributions which have a Lancaster expansion in terms of the reproducing kernel polynomials, which is a subclass of all bivariate distributions with multivariate Krawtchouk polynomial eigenfunctions, where eigenvalues only depend on the total degree of the polynomial eigenfunctions.

\begin{thm}
\begin{equation}
m(\boldsymbol{x},\boldsymbol{p})m(\boldsymbol{y},\boldsymbol{p})
\Big\{1 + \sum_{n=1}^N\rho_{n}Q_{n}(\boldsymbol{x},\boldsymbol{y};N,\boldsymbol{p}) \Big\} \geq 0
\label{joint:0}
\end{equation}
\emph{for all $\boldsymbol{x},\boldsymbol{y}$ if and only if}
\[
\rho_{n} 
=\mathbb{E}\Big [Q_n(Z;N,p)\Big ]
\]
\emph{for $p \geq 1-\min_{j\in [d]}p_j$, $n = 1,\ldots,N$ and some random variable $Z$ on $\{0,1,\ldots N\}$.}
\end{thm}

\begin{proof} {\bf Sufficiency.} This follows from the result that, with  $K(\boldsymbol{x},\boldsymbol{y},z)$ from (\ref{KK:0}),
\begin{equation}
m(\boldsymbol{x},\boldsymbol{p})m(\boldsymbol{y},\boldsymbol{p})K(\boldsymbol{x},\boldsymbol{y},z) 
\label{Kstar}
\end{equation}
is non-negative for $p \geq 1 - \min_{j\in [d]}p_j$ from Theorem  \ref{Duplicationthm}. Then a mixture of (\ref{Kstar}) with respect to a distribution on $z$ is also non-negative.
\medskip

\noindent
{\bf Necessity.} Suppose $(\boldsymbol{X},\boldsymbol{Y})$ has a joint distribution (\ref{joint:0}) and that without loss of generality $p_1=\min_{j\in [d]}p_j$.  Let ${X} = X_1$, ${Y} = Y_1$.\\
Then from Remark \ref{condense:rmk}, $({X},{Y})$ has a distribution
\begin{eqnarray}
&&{N\choose x}p_1^{x}(1-p_1)^{N-x}{N\choose y}p_1^{y}(1-p_1)^{N-y}
\nonumber \\
&&~~~~\times\Big \{ 1 +
\sum_{n=1}^N\rho_nh_n(p_1)Q_n({x};N,p_1)Q_n({y};N,p_1)\Big \}.
\label{XXX:0}
\end{eqnarray}
Setting $x=0$, $y=z$, $r=p_1$ in (\ref{XXX:0})
\begin{equation}
\sum_{n=0}^N\rho_nh_n(r)Q_n(z;N,r) \geq 0.
\label{XXX:1}
\end{equation}
Using the duplication formula (\ref{kduplicate:1}) with $s=p$, then multiplying by binomial marginals,
\begin{equation}
{N\choose x}s^x(1-s)^{N-x}
{N\choose y}s^y(1-s)^{N-y}
\sum_{n=0}^N\rho_nh_n(s)Q_n(x;N,s)Q_n(y;N,s) \geq 0
\label{XXX:2}
\end{equation}
for $1/2 \leq 1 - r \leq \min(s,1-s)$
 and is thus a probability distribution. In (\ref{XXX:2})
\begin{equation}
\rho_nQ_n(y;N,s) = \mathbb{E}\Bigl [Q_n(X;N,s)\mid Y=y\Bigr ].
\label{XX:2}
\end{equation}
Setting $Z$ to have the distribution of $X$ conditional on $Y=0$, and noting that $1-r \leq s$ is the same as $1-p_1 \leq p$ gives the necessity.
\end{proof}
\subsubsection*{An extreme point Ehrenfest Urn}\label{s:extreme}
We now describe an urn based discrete time Markov chain whose transition probabilities are that of $\boldsymbol{Y}$ given  $\boldsymbol{X}=\boldsymbol{x}$ where $(\boldsymbol{X},\boldsymbol{Y})$ has an extreme point distribution (\ref{joint:0}) with $\rho_{n} = Q_n(z;N,p)$. That is, the transition probabilities are
\begin{equation}
m(\boldsymbol{y},\boldsymbol{p})
\Big\{1 + \sum_{n=1}^N\rho_{n}Q_{n}(\boldsymbol{x},\boldsymbol{y};N,\boldsymbol{p}) \Big\}. 
\label{Markov:100}
\end{equation}
This chain has the kernel polynomials as eigenfunctions along with explicitly available eigenvalues. These are used to get sharp rates of convergence.
An urn has $N$ balls of $d$ colours labeled $1,\ldots, d$. A discrete time Markov chain, with state space $\boldsymbol{x} = (x_1,\ldots ,x_d),\>|\boldsymbol{x}|=N$, counting balls of colours in $\{1,\ldots,d\}$ is now constructed so that transition functions of the chain from $\boldsymbol{x} \to \boldsymbol{y}$ for fixed $z = 0,1,\ldots ,N$ are
$
m(\boldsymbol{y},\boldsymbol{p})K(\boldsymbol{x},\boldsymbol{y},z)
$,
with $K$ defined in (\ref{KK:0}), equivalently (\ref{Markov:100}).
In a transition choose $z$ balls at random without replacement from the urn to change colour independently such that if a ball of type $j \in [d]$ is chosen then a change is made to colour $k\ne j$ with probability $p_k/p$, or the colour is left unchanged as colour $j$ with probability $(p_j-q)/p$. Take $q \leq \min_{j\in [d]}p_j$
which ensures that $p_j/p \leq 1$ for all $j \in [d]$.
A colour change for a single ball of colour $j$ which is in the $z$ balls chosen occurs according to the \emph{pgf}
$\big [T_0(\boldsymbol{t}) - qt_j\big ]/p$.
It is now shown that if $\boldsymbol{X}$ has a $m(\boldsymbol{x},\boldsymbol{p})$ distribution then the joint \emph{pgf} of $(\boldsymbol{X},\boldsymbol{Y})$ is (\ref{jointt:1}). The conditional \emph{pgf} for the distribution of $\boldsymbol{Y}$ given $\boldsymbol{X}=\boldsymbol{x}$ is
\[
\sum_{|\boldsymbol{z}| = z}{\cal H}(\boldsymbol{z};\boldsymbol{x})p^{-z}
\prod_{j=1}^d\big [T_0(\boldsymbol{t}) - qt_j\big ]^{z_j}t_j^{x_j-z_j},
\]
where 
\[
{\cal H}(\boldsymbol{z};\boldsymbol{x}) =
\frac{
{x_1 \choose z_1}\cdots {x_d \choose z_d}
}
{ 
{N \choose z}
},
\]
a hypergeometric probability.
If $\boldsymbol{X}$ has a $m(\boldsymbol{x},\boldsymbol{p})$ distribution, then the joint \emph{pgf} of $(\boldsymbol{X},\boldsymbol{Y})$ is therefore
\begin{eqnarray*}
&&
\sum_{|\boldsymbol{x}| = N}m(\boldsymbol{x};N,\boldsymbol{p})
\sum_{|\boldsymbol{z}| = z}{\cal H}(\boldsymbol{z};\boldsymbol{x})p^{-z}
\prod_{j=1}^d\big [T_0(\boldsymbol{t}) - qt_j\big ]^{z_j}t_j^{x_j-z_j}s_j^{x_j}
\nonumber \\
&&=
p^{-z}
\sum_{|\boldsymbol{z}| = z}m(\boldsymbol{z};z,\boldsymbol{p})
\prod_{j=1}^d\big [T_0(\boldsymbol{t}) - qt_j\big ]^{z_j}s_j^{z_j}
\nonumber \\
&&~~~~~~~~~\times
\sum_{\boldsymbol{x}\geq \boldsymbol{z}}m(\boldsymbol{x}-\boldsymbol{z};N-z,\boldsymbol{p})\prod_{j=1}^d(s_jt_j)^{x_j-z_j}
\nonumber \\
&&~~=p^{-z}\Big [T_0(\boldsymbol{s})T_0(\boldsymbol{t}) - q \sum_{i=1}^dp_is_it_i\Big ]^z
\Big [\sum_{i=1}^dp_is_it_i\Big ]^{N-z},
\end{eqnarray*}
in agreement with an earlier calculation (\ref{jointt:1}) for the \emph{pgf} of
\[
m(\boldsymbol{x},\boldsymbol{p})
m(\boldsymbol{y},\boldsymbol{p})
K(\boldsymbol{x},\boldsymbol{y},z).
\]
\subsection{Chi-squared distance \label{chi:1000}}\label{s:cutoff}
\begin{Ex}\label{chiex}
 With all of the machinery in place, we offer an example of how kernel polynomials can be used to give sharp rates of convergence of a Markov chain on configurations to a multinomial stationary distribution. In addition to demystifying the notation, the example offers two surprises. First it shows a striking disparity between $\ell^1$ and $\ell^2$ convergence. One usual route to bounding $\ell^1$ (total variation) is to use Cauchy-Schwarz to bound 
  $\ell^1$ by $\ell^2$. This approach breaks down here. Second, it shows that a \emph{non-sticking} dynamics can speed up convergence. These surprises are explained after a careful statement of the main result followed by a proof and final remarks.
  
  The state space ${\cal X}$ consists of $N$ balls of $d$ possible colours in an urn. Let $\boldsymbol{p}= (p_1,\ldots,p_d)$ be fixed with $p_i >0, p_1+\cdots +p_d=1$ and let $m({\boldsymbol{n}}) = {N \choose n_1,\ldots,n_d}\prod_{i=1}^dp_i^{n_i}$ be the multinomial distribution on ${\cal X}$. To describe the Markov chain on ${\cal X}$, fix $0 \leq q \leq \min_{i\in [d]}p_i$ and let $p=1-q$. In words: \emph{pick one of the $N$ balls, uniformly at random. If it has colour $j$ change its colour to $i\ne j$ with probability $p_i/p$. Let it remain at colour $j$ with probability $(p_i-q)/p$}.  Let $p(\boldsymbol{x},\boldsymbol{y})$ be the transition kernel (chance of going from $\boldsymbol{x}$ to $\boldsymbol{y}$ in one step) and $p^l(\boldsymbol{x},\boldsymbol{y})$ be the chance after $l$ steps.
  This is a simple case of the extremal urns of Section \ref{section:4} with $z=1$. As shown there, $p(\boldsymbol{x},\boldsymbol{y})$ is a reversible ergodic Markov chain with stationary distribution $m(\boldsymbol{x})$. 
  
  The following result gives sharp upper and lower bounds on  
  $ \chi^2_{\boldsymbol{x}}(l) = \sum_{\boldsymbol{y}} \big (p^l(\boldsymbol{x},\boldsymbol{y}) - m(\boldsymbol{y})\big )^2/m(\boldsymbol{y})$, the chi-squared distance after $l$ steps starting from $\boldsymbol{x}$ for $\boldsymbol{x} = N\boldsymbol{e}_i=(0,\ldots, \overset{i}{N},\ldots 0)$ where all balls start in colour $i$. It shows a cutoff at
  \[
  l = \frac{Np}{2}\Bigg ( \log N\Big (\frac{1}{p_i}-1\Big )\Bigg ).
  \]
  
  \end{Ex}
  	\begin{thm}\label{pt:1}
  	For $\boldsymbol{x} = N\boldsymbol{e}_i$, $l = \frac{Np}{2}\Bigg (\log N \Big (\frac{1}{p_i}-1\Big ) +c \Bigg )$
  	\[
  	\Big (1-\frac{1}{Np}\Big )^{2l}N\Big (\frac{1}{p_i}-1\Big ) \leq \chi^2_{\boldsymbol{x}}(l) \leq e^{e^{-c}}-1.
  	\]
  	\end{thm}
  	\begin{rmk} \label{pr:1} To help parse these bounds, note that when $c$ is positive and large the right hand side is asymptotic to $e^{-c}$ and so exponentially small. When $c$ is negative and large, the left side is asymptotic to $e^{-c}$ and so exponentially large. Note that the bounds are absolute, uniformly in all parameters involved, so for an actual $N,\boldsymbol{p},q$ and $l$ (which determines $c$) one can simply calculate them. If numerical calculation is ever of interest, the proof below gives a simple useful closed form sum which can be easily computed.
  	\end{rmk}
  	\begin{rmk} \label{pr:2}
  	The walk shows a sharp chi-square cutoff at 
  	\[
  	l^* = \frac{Np}{2}\Bigg (\log N \Big (\frac{1}{p_i}-1\Big ) \Bigg ).
  	\]
  	Note that this depends on $p_i$. If the starting state has, for example,
  	$p_i=1/2$, order $N\log N$ steps are necessary and suffice. If the starting state has, for example, $p_i=1/2^N$, order $N^2$ steps are necessary and suffice.
  	\end{rmk}
  	\begin{rmk}  \label{pr:3}	In contrast to Remark \ref{pr:2} consider convergence in total variation ($\ell^1$, 
  	$\|p^l(\boldsymbol{x},\cdot)-m\|_{\text{TV}}
  	 = \frac{1}{2}\sum_{\boldsymbol{y}}|p^l(\boldsymbol{x},\boldsymbol{y})-m(\boldsymbol{y})|$ ). For simplicity take $q=0$ (so $p=1$). Thus once a ball has been hit at least once it has exactly the right distribution. Let $T$ be the first time all balls have been hit at least once. This is a strong stationary time and standard arguments, using the coupon collector's problem
  	 \citep{AD1986,LPW2008} 
  	 show
  	 \begin{thm} \label{pt:2}
  	 For any starting state $\boldsymbol{x}$ and all $\boldsymbol{p}$, for $l= N(\log N + c)$, $c>0$,
  	 \[
  	 \|p^l(\boldsymbol{x},\cdot)-m\|_{\text{TV}} \leq e^{-c}.
  	 \]
  	 \end{thm}
There is a matching lower bound if say $p_i$ is small, starting from all balls in state $i$. This disparity between different measures of convergence is unsettling, even on reflection.
If the starting state was $N\boldsymbol{e}_i$ with, for example, $p_i=1/2^N$, then the presence of a ratio in the chi-squared distance means that the number of steps $l$ must be very large to make it exponentially sure that all balls of colour $i$ have been hit at least once. A careful look at the coupon collector's bound shows that this requires order $N^2$ steps.
 One further note: suppose the chain starts as in Theorem \ref{pt:1}, with say $p_i=1/2$. From Theorem \ref{pt:1} using the Cauchy-Schwarz bound shows that $\|p^l(\boldsymbol{x},\cdot)-m\|_{\text{TV}} \leq e^{-c/2}$ for
  	 $l=\frac{N}{2}\Big (\log N +c \Big )$, which is a smaller bound than Theorem \ref{pt:2} gives.
  	 \end{rmk}
  	 \begin{rmk} \label{pr:4}
  	 A final aspect that we find surprising. Consider the effect of the parameter $p$ (or $q=1-p$). Is seems intuitive that setting $p=1$, so when a ball is hit it changes with the exact correctly distributed colour, should be optimal. The bound shows that decreasing the holding makes for faster mixing. For example suppose $p_i = \min_{j\in [d]}p_j$ and $q=p_i$. If the balls start in colour $i$ they never hold.
  	 \end{rmk}
  	 \begin{proof}{(Theorem \ref{pt:1}).}
  	 
  	 From (\ref{Markov:100}), for any starting state $\boldsymbol{x}$,
  	 \[
  	 \chi^2_{\boldsymbol{x}}(l)
  	 = \sum_{\boldsymbol{y}}\big (p^l(\boldsymbol{x},\boldsymbol{y}) - m(\boldsymbol{y})\big )^2/m(\boldsymbol{y}) 
  	 = \sum_{n=1}^N\rho_n^{2l}Q_n(\boldsymbol{x},\boldsymbol{x};N,\boldsymbol{p})
  	 \]
  	 with $\rho_n = 1 - n/Np$ by a simple calculation from the explicit form of the univariate Krawtchouk polynomial in Section \ref{section:1}. For $\boldsymbol{x} = N\boldsymbol{e}_i$, from (\ref{twocase:00})
  	 $Q_n(\boldsymbol{x},\boldsymbol{x};N,\boldsymbol{p}) = {N\choose n}\Big ( \frac{1}{p_i} -1\Big )^n$. 
    Thus, for this starting $\boldsymbol{x}$,
    \[
    \chi_{\boldsymbol{x}}^2(l) = \sum_{n=1}^N\Big (1 - \frac{n}{Np}\Big )^{2l}{N\choose n}\Big (\frac{1}{p_i}-1\Big )^n.
    \]
    For the upper bound, use $1-x \leq e^{-x}$ and ${N\choose n} \leq \frac{N^n}{n!}$ to see 
    \begin{eqnarray*}
   && \chi^2_{\boldsymbol{x}}(l) \leq \sum_{n=1}^N 
    \frac{1}{n!}\cdot \exp \Bigg \{\frac{-2ln}{Np} + n \log \Bigg (N\Big (\frac{1}{p_i} - 1 \Big )\Bigg )\Bigg \}
    \nonumber \\
    &&~= \sum_{n=1}^N\frac{e^{-nc}}{n!} \leq \sum_{n=1}^\infty \frac{e^{-nc}}{n!} = \exp \{e^{-c}\}-1.
    \end{eqnarray*}
    For the lower bound, just use the first term in the expression for $\chi^2_{\boldsymbol{x}}(l)$.
    \end{proof}
    The calculations above can be carried out for other starting configurations. For example if $d=N$ and $\boldsymbol{x} = (1,1,\ldots,1)$ (one ball of each colour),
    \[
    Q_n(\boldsymbol{x},\boldsymbol{x};N,\boldsymbol{p})  = \sum_{j=0}^n{N\choose j}{N-j\choose n-j}j!(-1)^{n-j}s_j(\boldsymbol{p}^{-1})/N_{[j]}^2
    \]
    with $s_j$ the $j$th elementary symmetric function, so
    $s_1(\boldsymbol{p}^{-1}) = \sum_{i=1}^d 1/p_i$,
     $s_j(\boldsymbol{p}^{-1}) = \sum_{1 \leq k_1< k_2 \cdots < k_j \leq d}
     1/(p_{k_1}\cdots p_{k_j})$.
     We have not carried out the details of bounding the convergence rate but observe that when $p_i = 1/N$, $1 \leq i \leq N$, 
     $Q_1(\boldsymbol{x},\boldsymbol{x};N,\boldsymbol{p}) = -N + s_1(\boldsymbol{p}^{-1})/N = 0$ so the second term must be used to get a lower bound.

It is natural to wonder what the right rate is for total variation convergence when $q > 0$. The stopping time argument given above breaks down then. The following calculations show that, $\ell^1$ and $\ell^2$ rates agree so that the stopping time argument is \emph{off} by a factor of 2, provided that $p_i$ is bounded away from zero. As shown above these two rates can be very different if $p_i$ is small. The techniques involved make a nice illustration of our theory. With notation as in Theorem \ref{pt:1}, let $\{\boldsymbol{Y}_k\}_{k=0}^\infty$ be the full multinomial chain, with  $\boldsymbol{Y}_0= N\boldsymbol{e}_i$.
Let $X:$ configuration space $\to \{0,1,\ldots,N\}$ denote the number of balls of colour $i$. 
 Let $X_k = X(Y_k)$. This is a birth-death chain (use Dynkin's criteria to see this) with a Binomial$(N,p_i)$ stationary distribution and transition density
\begin{eqnarray*}
K(j,j-1) &=&\frac{j}{N}\Big (1 - \frac{p_i-q}{p}\Big ) :=  \frac{j}{N}\cdot\alpha
\nonumber \\
K(j,j+1) &=& \Big ( 1-\frac{j}{N}\Big )\frac{p_i}{p} := \Big ( 1 - \frac{j}{N}\Big )\cdot \beta
\nonumber \\
K(j,j) &=& 1 - \frac{j}{N}\alpha - \Big ( 1 - \frac{j}{N}\Big )\beta.
\end{eqnarray*}
Note that $\alpha+\beta = 1/p$. 

The chain is also a particular case of the chain  with transition probabilities (\ref{Markov:100}) where $d=2$, $\rho_n = 1 - \frac{n}{Np}$ and the reproducing kernel polynomials split into Krawtchouk polynomial eigenfunctions, where $p_1,p_2$ in the notation are replaced by $p^\prime_1 = p_i$, $p_2^\prime = \sum_{j\ne i}p_j = 1 - p_i$. The other parameters $z=1$ and $p$ are unchanged.  The two types correspond to type $i$ balls and balls not of colour $i$ lumped together as a second type. The lumped chain is still Markov. In the following we proceed from a birth and death chain approach rather than appeal to the structure from (\ref{Markov:100}).

Note next that, starting from $X_0=N$, for any set of configurations $A$ for all $j$,
 $0 \leq j < \infty$,
\[
P\{Y_j \in A\mid X_j = a\} = m(A\mid X=a).
\]
That is, $X$ is a sufficient statistic for $\{{\cal L}(Y_j),  m\}$. Now theorem 6.1 in \citet{DZ1982} shows that for any $l$,
\[
\|P^l(Ne_i,\cdot)-m(\cdot)\|_{\text{TV}} = \|K(N,\cdot)- \mu(\cdot)\|_{\text{TV}}.
\]
A similar equality is shown to hold for other distances (the $f$ divergences).

The next observation is that the $\{X_i\}$ chain has Krawtchouk polynomial eigenfunctions with eigenvalues $\beta_a = 1 - \frac{a}{Np}$, $0 \leq a \leq N$. The proof of this follows from Cannings' criteria. The operator $K$ preserves degree $a$ polynomials
\[
\mathbb{E}\big [X_1^a\mid X_0=j\big ] = j^a\Big (1 - \frac{a}{Np}\Big ) + \text{lower~order~terms;}
\]
indeed 
\begin{eqnarray*}
&&\mathbb{E}\big [X_1^a\mid X_0=j] 
\nonumber \\
&~~~~~~~=& (j-1)^a\frac{j}{N}\alpha + 
j^a\Big (1 - \frac{j}{N}\alpha - \big (1 - \frac{j}{N}\big )\beta\Big ) + (j+1)^a\big (1 - \frac{j}{N}\big )\beta
\nonumber \\
&~~~~~~~=& \big (j^a - aj^{a-1} + \cdots\big )\frac{j}{N}\alpha  + (j^a + aj^{a-1} + \cdots  )\big (1 - \frac{j}{N}\big )\beta
\nonumber \\
&~~~~~~~~~~~~~~~~~~~~+& j^a\Big (1 - \frac{j}{N}\alpha - \big (1-\frac{j}{N}\big )\beta\Big )
\nonumber \\
&~~~~~~~=&
j^a\big (1 - \frac{a}{N}(\alpha + \beta)\big ) + \mathcal{O}(j^{a-1})
\nonumber \\
&~~~~~~~=& j^a\big (1 - \frac{a}{Np}\big ) +  \mathcal{O}(j^{a-1})
\end{eqnarray*}
The Krawtchouk polynomials are $Q_n(x;N,p_1)$ of Section \ref{section:1}. In particular 
\begin{eqnarray*}
Q_1(x;N,p_i) &=& \frac{1}{Np_i}(Np_i-x) 
\nonumber \\
Q_2(x;N,p_i) &=& {N\choose 2}^{-1}\Bigg ({N-x \choose 2}- x(N-x)\frac{q_i}{p_i} + \Big (\frac{q_i}{p_i}\Big )^2{x \choose 2} \Bigg ).
\end{eqnarray*}
To use the 2nd moment method we need to express $x^2$ as a linear combination of $Q_0,Q_1,Q_2$;
\begin{lem}
$x^2 = aQ_2 + bQ_1+c Q_0$ with $a= p_i^2N(N-1)$, $c=N^2p_i^2 + Np_iq_i$, $b = -2N^2p_i^2 - Np_i^2 - N p_iq_i$.
\end{lem}
\begin{proof}
Since $Q_n(0) = 1$ by construction, evaluating at $x=0$ gives $0=a+b+c$. From the definition of $Q_2$, the coefficient of $x^2$ is $\frac{1}{N(N-1)p_i^2}$ so $a=p_i^2N(N-1)$. Finally taking $x=Np_i$, yields $(Np_i)^2 = -a \frac{q_i}{(N-1)p_i}+c$. Solving these equations for $a,b,c$ yields the claimed expression.
\end{proof}
Observe next that, using $Q_n(N;N,p_i) = \big (-q_i/p_i\big )^n$,
\begin{eqnarray}
\mathbb{E}_N\big [X_l\big ] &=& \mathbb{E}_N\big [X_l - Np_i\big ] + Np_i
\nonumber \\
&=&
Np_i\mathbb{E}_N\big [-Q_1(X_l;N,p_1)\big ] + Np_i
\nonumber \\
&=& Nq_i\Big (1 - \frac{1}{Np}\Big )^l + Np_i
\label{PX:1} \\
\mathbb{E}\big [X_l^2\big ] &=& a\Big (1 - \frac{2}{Np}\Big )^l \Big (\frac{q_i}{p_i}\Big )^2 - b\Big (1 - \frac{1}{Np}\Big )^l\frac{q_i}{p_i} + c
\label{PX:2}
\end{eqnarray}
This allows computation of $\text{Var}_N(X_l)$. In the computations that follow, the assumption that $p_i$ is bounded below ensures $q_i/p_i = \mathcal{O}(1)$.  Observe that 
$a = (Np_i)^2 + \mathcal{O}(N)$, $b= -2 (Np_i)^2 + \mathcal{O}(N)$, $c= (Np_i)^2 + \mathcal{O}(N)$. Using this, (\ref{PX:1}) and (\ref{PX:2}) show
\begin{eqnarray*}
\text{Var}_N(X_l) &=& \mathbb{E}_N\big [X_l^2\big ] 
- \mathbb{E}_N\big [X_l\big ]^2
\nonumber \\
&=& N^2q_i^2 \Bigg (\Big (1-\frac{2}{Np}\Big )^l - \Big (1 - \frac{1}{Np}\Big)^{2l}\Bigg ) 
+ \mathcal{O}(N)
\nonumber \\
&=& \mathcal{O}(N)
\end{eqnarray*}
The implicit constant in $\mathcal{O}$ is uniformly bounded. The standard deviation of $X_l$ is $\mathcal{O}(\sqrt{n})$.

\noindent
From (\ref{PX:1}), for $l=\frac{Np}{2}\Bigg(\log N\Big (\frac{1}{p_i}-1\Big )+c\Bigg )$,  $c \in \mathbb{R}$
\[
\mathbb{E}_N\big [X_l\big ] = Nq_i\Big (1 - \frac{1}{Np}\Big )^l + Np_i = Np_i + e^{-c}\sqrt{Np_iq_i}.
\]
If $c$ is negative and large, $X_l$ is concentrated many standard deviations away from $Np_i$. On the other hand, the binomial stationary distribution is concentrated around $Np_i$, with standard deviation $\sqrt{Np_iq_i}$. This implies that for 
$l=\frac{Np}{2}\Bigg(\log N\Big (\frac{1}{p_i}-1\Big )+c\Bigg )$, $c <0$, the total variation distance between $X_l$ and the binomial $(N,p_i)$ is large. These calculations are summarized:
\begin{thm}
For the multinomial Markov chain of Theorem \ref{pt:1}, starting at $N\boldsymbol{e}_i$, with 
$l=\frac{Np}{2}\Bigg(\log N\Big (\frac{1}{p_i}-1\Big )+c\Bigg )$, $c \in \mathbb{R}$ fixed
\[
f(c) \leq \|p_{N\boldsymbol{e}_i}^l - m\|_{\text{TV}} \leq e^{e^{-c}}-1
\]
with $f(c)$ bounded away from $0$ as $c \searrow -\infty$ provided $p_i$ is bounded away from $0$.
\end{thm}

\begin{rmk}
A general extreme point Markov chain has transition functions (\ref{Markov:100}) with $\rho_n=Q_n(z;N,p)$ for $0 < z \leq N$ and $p \geq 1 - \min_{j\in [d]}p_j$.
Recall a description of this chain. \emph{Pick $z$ of the $N$ balls, uniformly at random without replacement. If a ball chosen has colour $j$ change its colour to $i\ne j$ with probability $p_i/p$. Let it remain at colour $j$ with probability $(p_i-q)/p$}.

 We give useful upper and lower bounds on  $\chi^2$ convergence. 
\begin{thm}
For an initial condition $\boldsymbol{x}\ne N(1/d,\ldots,1/d)$ and $z/N \ne p$, let
\[
l = \frac{1}{2\big (1 - \big |1 - \frac{z}{Np}\big |\big )}
\Bigg ( \log N\Big ( \frac{1}{\min_{j\in [d]}p_j} -1\Big ) + c\Bigg ),
\] 
then
\begin{equation} 
\Big ( 1 - \frac{z}{Np}\Big )^{2l}\Big (\frac{1}{N}\sum_{j=1}^Np_j^{-1}x_j^2 - N\Big )
\leq \chi^2_{\boldsymbol{x}}(l)
\leq e^{e^c}-1.
\label{bigcutoff:0}
\end{equation}
\end{thm}
There is a mixing speed trade-off between $p$ and $z/N$. The behaviour of $l$ can be very different to the case when $z=1$. 
If $z$ is held constant then 
\[
l = \frac{Np}{2z}
\Bigg ( \log N\Big ( \frac{1}{\min_{j\in [d]}p_j} -1\Big ) + c\Bigg ).
\]
If $z=[N\alpha]$ then
\[
l = \frac{1}{2\big (1 - \big |1 - \frac{\alpha}{p}\big |\big )}
\Bigg ( \log N\Big ( \frac{1}{\min_{j\in [d]}p_j} -1\Big ) + c\Bigg )
\]
with
\[
l = \frac{1}{2}
\Bigg ( \log N\Big ( \frac{1}{\min_{j\in [d]}p_j} -1\Big ) + c\Bigg )
\] 
if $z=[Np]$.

\begin{proof}
The proof is very similar to that in Theorem \ref{pt:1} after finding the asymptotic form of $|Q(z;N,p)|$ as $N \to \infty$. The form is
\begin{eqnarray}
|Q_n(z;N,p)| &\sim&
 \Big |\frac{z}{Np}-1\Big |^n.
\nonumber \\
\label{qest:0}
\end{eqnarray}
This result is easily seen from the generating function (\ref{KGF:00}). Replacing $t$ by $t/N$
\begin{eqnarray*}
\sum_{n=0}^N{N\choose n}N^{-n}t^nQ_n(z;N,p) &=& \Big (1-\frac{t}{N}\frac{q}{p}\Big )^{z}\Big ( 1 + \frac{t}{N}\Big )^{N - z}
\nonumber \\
&\sim& \exp \Big \{t \Big (1 - \frac{z}{Np}\Big )\Big \}.
\end{eqnarray*}
The left side of the generating function is asymptotic to 
\[
\sum_{n=0}^N\frac{t^n}{n!}Q_n(z;N,p),
\]
so equating coefficients of $t^n$,
\[
Q_n(z;N,p) \sim \Big (1 - \frac{z}{Np}\Big )^n,
\]
which is in $[-q/p,1]$. Taking the absolute value gives (\ref{qest:0}).
An inequality needed is that for any starting configuration $\boldsymbol{x}$, from (\ref{twocase:2})
\begin{equation}
|Q_n(\boldsymbol{x},\boldsymbol{x};N,\boldsymbol{p})| \leq {N\choose n}\Big (\frac{1}{\min_{j\in [d]}p_j}-1 \Big )^n.
\label{qest:1}
\end{equation}
Then
\begin{eqnarray}
\chi_{\boldsymbol{x}}^2(l) &=& \sum_{n=1}^NQ_n^{2l}(z;N,p)Q_n(\boldsymbol{x},\boldsymbol{x};N,\boldsymbol{p})
\nonumber \\
&\lesssim&
\sum_{n=1}^N{N\choose n}\Big (1-\frac{z}{Np}\Big )^{2ln}
\Big (\frac{1}{\min_{j\in [d]}p_j}-1 \Big )^n
\nonumber \\
&\leq& \exp \Bigg \{e^{-2l\Big (1 - \Big |1 - \frac{z}{Np}\Big |\Big )+ \log N\Big (\frac{1}{\min_{j\in [d]}p_j}-1 \Big )}\Bigg \} -1
\nonumber \\
&=& e^{e^{-c}}-1.
\label{qest:2}
\end{eqnarray}
The left side bound in (\ref{bigcutoff:0}) is the first term in the first line expansion of (\ref{qest:2}). If $c$ is large and positive the right side of (\ref{bigcutoff:0}) is exponentially small. For the left side, which is positive if
$\boldsymbol{x}\ne N(1/d,\ldots,1/d)$ and $z/N \ne p$,

\begin{eqnarray*}
\chi_{\boldsymbol{x}}^2(l)&\geq &\Big ( 1 - \frac{z}{Np}\Big )^{2l}\Big (\frac{1}{N}\sum_{j=1}^Np_j^{-1}x_j^2 - N\Big )
\nonumber \\
 &=& \Big |1 - \frac{z}{Np}\Big |^{
\frac{1}{1 - \big |1 - \frac{z}{Np}\big |}\Bigg ( \log N \Big (\frac{1}{\min_{j\in [d]}p_j}-1\Big ) + c\Bigg )}\Big (\frac{1}{N}\sum_{j=1}^Np_j^{-1}x_j^2 - N\Big )
\nonumber \\
&\geq & 
(1-u)^{\frac{c}{u} +  \frac{1}{u}\log N \Big (\frac{1}{\min_{j\in [d]}p_j}-1\Big )}
\cdot N\Big (\frac{1}{d\max_{j\in [d]}p_j}-1\Big ),
\end{eqnarray*}
where $u = 1 - \big |1 - \frac{z}{Np}\big |$. 
If $c$ large and negative then this bound is large.
\end{proof}
When $p=1$, the same coupon collector's bound (the first time all the colours of balls have been hit) when $z$ are removed and replaced in a transition is $(N/z)\log N + c$  works uniformly in $p_i$ as above.
\end{rmk}

\begin{rmk}
  We have been mystified by the high multiplicity of eigenvalues in the extremal urn models described above. Usually, multiplicity of eigenvalues comes from having an underlying symmetry, a group acting on the state space preserving transition probability \citep{BDPX2005}. We do not see such symmetry in, for example, the model treated in Example \ref{chiex}. There \emph{is} a conceptual explanation that is quite different than symmetry. To set things up, consider a Markov chain on $N$ copies of $\{1,2,...,d\}$. With $p_i$ fixed, this chain will have product measure as it's stationary distribution. The dynamics are as follows: pick one of the $N$ coordinates at random and change the colour of the coordinate as per Example \ref{chiex}.  Now the symmetric group ${\cal S}_N$ acts on the state space and the transitions are symmetric with respect to this. The orbit chain is our chain on multinomial configurations. It is easy to diagonalize the lifted chain and the eigenvalues of the lumped chain must be among those of the lifted chain. The eigenvalues of such product chains are simply $1/N$ times a sum of the eigenvalues of the coordinate chain (repetitions allowed). For this example, the coordinate chain is $(1/p)$ times a matrix with all rows the stationary distribution plus $(1-1/p)$ times the identity matrix. Thus the coordinate chain has  eigenvalues $1$ and $1-1/p$ (with multiplicity $d-1$). From this it follows that the eigenvalues of the product chain are $\big (j +(N-j)(1-1/p)\big )/N$ with $j= 0,1, \ldots, N$, an extremely limited set. Thus the chain on configuration space has high multiplicity of it's eigenvalues. The upshot of all this is that an explanation of 'kernel eigenfunctions' is tied to degeneracy of the coordinate chain, not symmetry.
\end{rmk}
\subsection*{Acknowledgement}
Persi Diaconis research was partially funded by DMS 08-04324.

We thank Jimmy He for a careful reading of the manuscript and his corrections. 

Two referees are thanked for their corrections, comments and suggestions.

\end{document}